\documentclass{amsart}
\usepackage[margin=2.5cm]{geometry}
\usepackage{amsmath,amssymb,amsfonts,amstext,amsthm,bm,amsrefs} 
\usepackage{mathrsfs}
\usepackage{graphicx}
\usepackage{tikz}
\usetikzlibrary{arrows,decorations.pathmorphing,decorations.markings,matrix,positioning}
\tikzset{
	closed/.style = {decoration = {markings, mark = at position 0.5 with { \node[transform shape, xscale = .8, yscale=.4] {/}; } }, postaction = {decorate} },
	open/.style = {decoration = {markings, mark = at position 0.5 with { \node[transform shape, scale = .7] {$\circ$}; } }, postaction = {decorate} },
	commutative diagrams/.cd,
	mysymbol/.style={start anchor=center, end anchor=center,draw=none}
}
\usepackage{tikz-cd}
\usepackage[all,cmtip]{xy}
\usepackage{xcolor}
\usepackage{enumitem}
\usepackage[pdftex,pdfpagelabels,bookmarks,hyperindex,hyperfigures]{hyperref}
\usepackage[capitalize]{cleveref}
\newtheorem{thmx}{Theorem}
\newtheorem{propx}{Proposition}

\newtheorem{corollary}{Corollary}[section]
\newtheorem{theorem}[corollary]{Theorem}

\newtheorem{lemma}[corollary]{Lemma}
\newtheorem{proposition}[corollary]{Proposition}
\theoremstyle{definition}
\newtheorem{convention}[corollary]{Convention}
\newtheorem{definition}[corollary]{Definition}

\newtheorem{notation}[corollary]{Notation}
\newtheorem{propositiondefinition}[corollary]{Proposition/Definition}
\newtheorem{remark}[corollary]{Remark}

\newtheorem{example}[corollary]{Example}

\usepackage{faktor}

\newcommand{\id}{\operatorname{Id}}

\newcommand{\pp}{\mathfrak{p}} 

\newcommand{\frakc}{\mathfrak{c}}

\newcommand{\Mod}{\mathrm{Mod}} 
\newcommand{\SMod}{\mathrm{SMod}} 

\newcommand{\Ring}{\mathrm{Ring}} 
\newcommand{\SRing}{\mathrm{SRing}} 
\newcommand{\Set}{\mathrm{Set}} 

\newcommand{\Alg}{\mathrm{Alg}} 
\newcommand{\B}{\mathbb{B}} 
\newcommand{\N}{\mathbb{N}} 
\newcommand{\Q}{\mathbb{Q}} 
\newcommand{\R}{\mathbb{R}} 
\newcommand{\Z}{\mathbb{Z}} 

\newcommand{\Cong}{\mathrm{Cong}}

\newcommand{\calC}{\mathcal{C}}
\newcommand{\calF}{\mathcal{F}}
\newcommand{\calG}{\mathcal{G}}
\newcommand{\calO}{\mathcal{O}}
\newcommand{\calU}{\mathcal{U}}
\newcommand{\idm}{\mathrm{Idm}}

\newcommand{\Spec}{\mathrm{Spec}} 

\makeatletter

\@namedef{subjclassname@2020}{%
	\textup{2020} Mathematics Subject Classification}
\makeatother

\title{Idempotentization of Affine Schemes and Sheaves}
\author{F\'elix Baril Boudreau$^{\ast}$}
\address{Universit\'e du Luxembourg, D\'epartement de Math\'ematiques, Maison du nombre, 6, avenue de la Fonte, Esch-sur-Alzette, 4364, Luxembourg}
\email{felix.barilboudreau@uni.lu}
\thanks{$^{\ast}$Corresponding author.}
\author{Cristhian Garay L\'opez}
\address{Centro de Investigaci\'on en Matem\'aticas, A.C. (CIMAT). Jalisco S/N, Col. Valenciana CP. 36023 Guanajuato, Gto, M\'exico}
\email{cristhian.garay@cimat.mx}

\begin{document}
	
	
	
	\begin{abstract}
		In this article, we introduce the idempotentization process, which bears some philosophical and mathematical similarities with modern analytification and tropicalization. Idempotentization associates to any affine scheme an idempotent version of itself with respect to a fixed covering by distinguished affine open subschemes. Once this cover is fixed, we can functorially associate a Zariski sheaf of rings or modules to a sheaf of idempotent semiring or a sheaf of idempotent semimodules. We show that idempotentization is independent of the chosen cover and in the Noetherian case, the idempotentization of the structure sheaf recovers the global sections. Underlying our formalism is a combinatorial reflection of lattices of subobjects of ordered-theoretic objects seen as lattices coming from commutative algebra. This has topological consequences for the semiring of subtractive ideals of a commutative semiring $S$: On one hand, it is a topological retract of the semiring of congruence relations of $S$ for the coarse lower topology. On the other hand, it is a topological retract of the semiring of ideals of $S$ for the coarse upper topology.
	\end{abstract}
	\maketitle
	
	\section{Introduction}
	This article is the first step in exploring the idea of transforming algebraic geometry objects into idempotent semi-algebraic-geometry-like objects and studying their properties in an \emph{idempotentized} framework.  Our approach stems from exploiting correspondences between order-theoretic structures and idempotent algebraic structures and from considering the structure of certain non-Archimedean seminorms associated with commutative rings and modules over them. As a result, giving any affine scheme with a fixed covering by distinguished affine open subschemes, we \emph{idempotentize} the scheme and functorially idempotentize any Zariski sheaf of rings or modules defined over it, with respect to that data.
	
	\subsection{Overview of the Main Results}\label{Subsection_Overview}
	For a commutative ring $R$, and a module $M$ over it,  we write $\Mod(M, R)$ and $\Mod(M, R)^c$ for the set of sub-$R$-modules of $M$ and for its subset of those that are finitely generated. In particular, we set $\id(R) := \Mod(R,R)$ and $\id(R)^c := \Mod(R,R)^c$. These two sets can be endowed with structures of semirings with respect to the sum and product of ideals. They are  \emph{idempotent} because any ideal $I$ of $R$ satisfies $I + I = I$ and are in fact \emph{simple} because $I + R = R$. We consider on $\id(R)$ and on its subset $\id(R)^c$ the order $a \leq b$ given by $a + b = b$. This makes the map $u_R: R \to \id(R)^c$ that sends $a \in R$ to the principal ideal $R \cdot a$ into an integral seminorm, that is, it for each $a \in R$, we have $u_R(a) \leq R$. J. Giansiracusa and N. Giansiracusa \cite{Giansiracusa_Giansiracusa_2016}*{Proposition 2.5.4} and Macpherson \cite{Macpherson_2020}*{p.449, Example 3.16} observed that every seminorm that takes values in an idempotent semiring factors uniquely through $u_R$, making the latter the universal one. They deduced that the set-valued functor from the category of idempotent semirings that sends an idempotent semiring $S$ to the set of integral seminorms $R \to S$ is represented by $\id(R)^c$. A ``tropical'' analogue is also obtained in \cite{Jun_Mincheva_Tolliver_2024}*{Proposition 6.7}.
	
	We consider the set $\Mod(M, R)^c$ endowed with the structure of idempotent semimodule on $\id(R)^c$ coming from the multiplication of ideals on modules and obtain a universal seminorm in the following sense. If $v$ is a normed ring from a ring $R$ to a semiring $S$, $M$ an $R$-module and $N$ an $S$-semimodule, we say, following Berkovitch \cite{Berkovich_1990}*{p. 12}, that a seminorm of Abelian groups $w : M \to N$ is a \emph{$v$-seminorm} if there is a non-zero $C \in S$ such that
	$$
	w(r \cdot m) \leq C v(r) \cdot w(m),
	$$
	for all $r \in R$ and $m \in M$. The map $u_M : M \to \Mod(M,R)^c$ that sends $m \in M$ to the cyclic $R$-submodule $R \cdot m$ of $M$ is a $u_R$-norm. Indeed, for all $r \in R$ and $m \in M$, we have $u_M(r \cdot m) = u_R(r)u_M(m)$. The map $u_M$ is the universal $u_R$-seminorm. More precisely, we show the following (see Proposition \ref{prop_snorm_module}).
	\begin{propx}
		For an $R$-module $M$, the  map $u_M$ is an $u_R$-norm satisfying the following properties:\begin{enumerate}
			\item[(i)] Any $u_R$-seminorm on $M$ factors uniquely through $u_M$.
			\item[(ii)] If $f:M_1\xrightarrow{}M_2$ is an $R$-module morphism and $\Mod^c(f) : \Mod(M,R_1)^c \to \Mod(M,R_2)^c$ is the induced map, then $u_{M_2} \circ f = \Mod^c(f) \circ u_{M_1}$.
			\item[(iii)] For an $R$-module $M$, the functor $\text{sn}_{u_R}$ from the category of $\id(R)^c$-semimodules to the category of sets that sends a $\id(R)^c$-semimodule $N$ to the set  $\text{sn}_{u_R}(N)=\{u_R-\text{seminorms } w: M \to N\}$ is represented by $\Mod(M,R)^c$.
		\end{enumerate}
	\end{propx}
	
	Recall that the $u_R$-norm $u_M$ yields an ordered set bijection between the ordered set $\mathrm{Mod}(M)$ of $R$-submodules of $M$ and the ordered set $\mathrm{SMod}_k(\Mod(M,R)^c)$ of subtractive semimodules (or $k$-semimodules) of $\Mod(M,R)^c$ (see Proposition \ref{Correspondence_Modules_over_a_Ring}). Order theory implies in the case $M = R$, the above ordered bijection preserves primary ideals and radicals (see Propositions \ref{Correspondence_Ideals_Preserves_Primes} and \ref{proposition_radical_ideals}). This plays a role in our study of the prime $k$-spectrum $\Spec_k(\id(R)^c)$ of $\id(R)^c$. 
	
	We also make the following apparently new observations. Let $S$ be a semiring, and consider its semiring of ideals $\id(S)$, its semiring of $k$-ideals $\id_k(S)$ and its semiring of semiring congruence relations $Cong(S)$. These are posets for the set inclusion $\subseteq$.
	
	\begin{thmx}[Theorem \ref{Retraction_Theorem}]
		Endow the posets $(Cong(S), \subseteq)$ and $(Id_k(S), \subseteq)$ with the coarse lower topology, and define maps
		\[
		\begin{tikzcd}
			\left( Id_k\left( S \right), \subseteq \right) \arrow[rr, "\mathfrak{c}", shift left=4] &  & \left( Cong\left( S \right), \subseteq \right) \arrow[ll, "\mathfrak{r}", shift left=2],
		\end{tikzcd}
		\]
		by setting $\frakc(I)$ to be the semiring congruence generated by $\{ (a, 0) \mid a \in I \}$ for $I \in Id_k\left( S \right)$, and for $Y$ in $ Cong(S)$, we let $\mathfrak{r}(Y)$ be set $\{ a \in S : (a,0) \in Y\}$ . Then the map $\mathfrak{r}$ is a retraction of topological spaces.
	\end{thmx}
	
	\begin{thmx}[Theorem \ref{Retraction_Theorem_ideales}]
		Endow the posets $(Id(S), \subseteq)$ and $(Id_k(S), \subseteq)$ with the coarse upper topology, and define maps
		\[
		\begin{tikzcd}
			(Id_k\left( S \right),\subseteq) \arrow[rr, "\mathfrak{i}", shift left=4] &  &(Id\left( S \right),\subseteq)  \arrow[ll, "\mathfrak{j}", shift left=2],
		\end{tikzcd}
		\]
		by writing $\mathfrak{i}$ for the set theoretical inclusion and setting $\mathfrak{j}(I) = \langle I \rangle_k$ for the $k$-closure of $I \in Id(S)$. Then, $\mathfrak{j}$ is surjective and a retraction of topological spaces.
	\end{thmx}
	
	In Section \ref{Section_Idempotentization}, we build functors that allow us to transform schemes and sheaves of them into idempotent (and thus combinatorial) versions of themselves. We summarize the main features of this section below.
	
	\begin{thmx}
		Let $X$ be a scheme and let $\calF$ be a presheaf of rings on $X$, respectively of $\calO_X$-modules. Then, there is a sheaf $\Theta(\calF)$ of $\mathbb{B}$-algebras on $X$, respectively of $\Theta(\calO_X)$-semimodules. Now, suppose that $X = \Spec(R)$ is an affine scheme. Then,
		\begin{enumerate}
			\item[(i)] Its $R$-idempotentization is a pair comprised of the topological space $\Spec_k(\id(R)^c)$ and of a structure sheaf $\calO_{\Spec_k(\id(R)^c)}$ of $\mathbb{B}$-algebras.
			\item[(ii)] If $R$ is Noetherian, then $\id(R) = \id(R)^c$ and the global sections of $\calO_{\Spec_k(\id(R))}$ are $\id(R)$.
			\item[(iii)] Sheaf of rings and of modules over $\Spec(R)$ can be idempotentized with respect to any affine covering $\calU$ by distinguished open subsets of $\Spec(R)$. Moreover, the idempotentization is independent of the choice of $\calU$ and is functorial both with respect to the sheaves and to $\calU$.
		\end{enumerate}
	\end{thmx}
	
	\subsection{Related Works}
	In classical tropical geometry, one consider non-trivial valuations from an algebraically closed field $K$ to the tropical semiring $\mathbb{T}$ or more generally, as initiated by Aroca \cite{Aroca_2010}, to idempotent semirings of the form $\Gamma \cup \{\infty\}$, where $\Gamma$ is a totally ordered group. Since these semirings are never simple, while the semirings $\id(R)^c$ always are, the idempotentization approach is distinct from the classical one.
	
	The definition and study of the prime spectrum of $k$-ideals, which goes back to Lescot in (\cite{Lescot_2012}), is ubiquitous in our work since the $R$-idempotentization of an affine scheme is always the prime spectrum of $k$-ideals of some idempotent semiring. 
	
	In \cite{Rump_Yang_2008}*{Corollary 3}, Rump and Yang prove a conjecture of Anderson (see \cite{Anderson_1989}), which refines the Jaffard-Kaplansky-Ohm correspondence. They then extend the correspondence to sheaves on spectral spaces and prove that the spectrum of a B\'ezout domain and the spectrum of its corresponding Abelian $\ell$-group are in Hochster duality. Idempotentization allows a reinterpretation of one of their central constructions (see Examples \ref{Example_Bezout} and \ref{Example_Sheaf_Reinterpretation_Rump_Yang}). In connection with Hochster duality, Jun and al. proved in \cite{Jun_Ray_Tolliver_2022}*{Theorem 3.43} that any spectral space is homeomorphic to the prime $k$-spectrum of an idempotent semiring.
	
	In that work, the authors also consider the interplay between $k$-ideals and congruences in the spaces they build, but not study seek to determine the existence of any topological retractions as we do in our Theorems \ref{Retraction_Theorem} and \ref{Retraction_Theorem_ideales}. Congruences also appear in the study of supertropical structures by Izhakian and Rowen in \cite{Izhakian_Rowen_2016}, where they establish the Zariski correspondence between congruences of tropical polynomials and algebraic sets.
	
	Now, suppose that $K$ is an algebraically closed field that is complete with respect to some non-trivial non-Archimedean valuation to $\R^\times$. Let $X$ be a quasi-projective variety defined over $K$ together with a closed embedding into a projective toric variety $Y$. Payne defines the tropicalization of $X$ as the closure of the image of $X(K)$ in the tropicalization of $Y$ and then a tropicalization functor from toric embeddings to topological spaces. He proves that the Berkovitch analytification of $X$ is homeomorphic to the projective limit of the system formed by all the tropicalizations of $X$ with respect to these toric embeddings (see \cite{Payne_2009}*{Theorem 4.2}. Kajiwara defines a similar tropicalization functor in \cite{Kajiwara_2008}.
	
	Yue Ren (private communication) asked the first author if the idempotentization process could be made independent of the covering by using colimits as in Payne's work \cite{Payne_2009}. This question inspired us to prove that our idempotentization was independent of the choice of cover by distinguished open affine subsets (see Subsection \ref{Subsection_Step5}).
	
	In \cite{Giansiracusa_Giansiracusa_2016}*{Section 6}, J. Giansiracusa and N. Giansiracusa developed the notion of scheme-theoretic tropicalization of a locally integral $\mathbb{F}_1$-scheme with respect to a valuation $v$ from a ring $R$ into an idempotent semiring $S$ (see \cite{Giansiracusa_Giansiracusa_2016}*{Definition 2.5.1}). This involves the so-called bend relations which are themselves (see \cite{Giansiracusa_Giansiracusa_2016}*{Proposition 6.1.1 and Remark 6.1.2}) defined with respect to the \emph{tropical ideals} introduced by Maclagan and Rinc\'on in \cite{Maclagan_Rincon_2018}*{Definition 1.1}. Building on this work, they define in \cite{Giansiracusa_Giansiracusa_2022}, for a scheme $X$ defined over a ring $R$, the universal embedding $X \hookrightarrow \widehat{X}$, which is universal among embeddings over $\Spec(R)$ into schemes equipped with a locally integral model over $\mathbb{F}_1$. They define the \textit{universal tropicalization of $X$} as the scheme-theoretic tropicalization of $X$ with respect to the canonical closed embedding $X \to \widehat{X}$ (see \cite{Giansiracusa_Giansiracusa_2022}*{Definition 3.5.1}). As an example of scheme-theoretic tropicalization, consider a field $K$ with a surjective valuation $v: K \to \mathbb{T}$ of valuation ring $\calO_v$. The universal property of the monoid algebra $K[R]$ applied to the identity map $1_R$ gives a surjective morphism $K[R] \to R$. The scheme-theoretic tropicalization of $\Spec(R)$ is $\Spec(\Mod(R,\calO_v)^c)$, where $\Mod(R,\calO_v)^c$ is the semiring of finitely generated $\calO_v$-submodules of $R$ (see \cite{Jun_Mincheva_Tolliver_2024}*{Examples 6.12 and 6.13}).
	
	The scheme-theoretical tropicalization of an affine scheme has similarities with the idempotentization of the topological space of that scheme. However, novelties and special features of our work are that we can idempotentize sheaves over an affine scheme and resulting objects are sheaves that have explicit descriptions.
	
	In \cite{Macpherson_2020}, Macpherson describes an algebraic-geometry-type theory of skeleta to study in a unified setting tropical varieties, skeleta of non-Archimedean analytic spaces, and affine manifolds with singularities. Besides having different motivations and main results, we note that most of Macpherson's article focuses on the case where the $R$-algebra $M$ is a non-Archimedean ring. He introduces his notion of adic space, which is a modified version of the concept introduced by Fujiwara and Kato \cite{Fujiwara_Kato_2018}*{\S II.2} in a way, as he mentions, that his adic spaces generalize formal schemes (see \cite{Macpherson_2020}*{p.441, Section 2.2}). In the direction of tropical adic spaces, Friedenberg and Mincheva focus in \cite{Friedenberg_Mincheva_2024},  on building a space out of congruences on an idempotent semiring. In terms of valuations, they consider those valued in totally ordered semifields that are not isomorphic to the Boolean semifield $\mathbb{B}$.
	
	In \cite{Lorscheid_2023}, Lorscheid unifies key aspects of the above works of Payne and Kajiwara, J. and N. Giansiracusa and Macpherson using his theory of ordered blueprints and ordered blue schemes. For example, he observes (see \cite{Lorscheid_2023}*{Subsections 13.1 and 13.2}) that the universal tropicalization of \cite{Giansiracusa_Giansiracusa_2022} and the analytification approach of \cite{Macpherson_2020} are essentially equivalent.
	
	\subsection{Structure of this Article}
	In Section \ref{Section_Prerequisites} we introduce the necessary background material about semirings, semimodules and order theory and illustrate it with relevant examples. We conclude that section with definitions and examples of non-Archimedean seminorms. Then, Section \ref{Section_Universal} discusses the universal valuation and our proof of the corresponding universal seminorm result. We also use the correspondence between submodules of a module $M$ and $k$-subsemimodules of the idempotentization of $M$ to recover algebraic consequences and their reinterpretations for the Zariski topology. We also prove our topological retraction theorems. We finish with Section \ref{Section_Idempotentization} where we describe the idempotentization of sheaves.
	
	
	\section{Prerequisites}\label{Section_Prerequisites}
	
	In this section, we recall the necessary background material and in particular introduce the basic objects that we study in this paper. The first subsection introduces definitions and properties of semirings and semimodules over them. In conjunction with this, we recall basic notions of order theory. All these concepts will be illustrated with relevant examples for the with comes in later sections. In the second subsection, we briefly recall properties of non-Archimedean seminorms defined over Abelian groups, rings and modules over normed rings. No new results are proven in that section. Readers acquainted with the theory might want to skim through it to be aware of the notation we use and our conventions and then go directly to the next section.
	\begin{notation}
		In this paper, the following symbols always have the same meaning. We write $R$ for a commutative ring with one, $M$ for an $R$-module, $S$ for a commutative semiring with one, $N$ for an $S$-semimodule, $\langle \mathcal{S} \rangle$ for the ideal or (semi)module generated by a set $\mathcal{S}$, $\B$ for the Boolean semiring $\{0,1\}$, $\N$: the additive monoid of nonnegative integers, $\Z$ for the ring of integers, $\Q$ for the field of rational number and $S-\Alg$ for the category of $S$-algebras, with $S \in \{ \B, \N \}$.
	\end{notation}
	
	\subsection{Semistructures and Order Theory}\label{Subsection_Semistructures_and_Order_Theory}
	
	Many references exist on semistructures and order theory. For example, the reader can consult \cite{Birkhoff_1948} for the former and \cite{Golan_1999} for the latter. All our structures (monoid, semiring, semimodule) are commutative and we will not mention it again to simplify the text.
	
	\begin{definition}
		A \emph{monoid} is a pair $(T, \star)$, where $T$ is a non-empty set together with an associative and commutative operation $\star$ and an element $e \in T$ satisfying $t \star e = t = e \star t$ for all $t \in T$, called \emph{identity} element. An element $t \in T$ is said to be idempotent if $t \star t = t$. The monoid $(T,\star)$ is said to be \emph{idempotent} if all its elements are idempotent.
	\end{definition}
	The element $e$ of $T$ can be seen to be unique.
	
	\begin{definition}
		A monoid $(T,\star)$ is said to be partially ordered if there is a partial order $\leq$ defined on $T$ such that for $t_1, t_2, t_3 \in T$, the condition $t_1 \leq t_2$ implies that $t_1 \star t_3 \leq t_2 \star t_3$ holds.
	\end{definition}
	\begin{definition}
		A \emph{semiring} is a non-empty set $S$ equipped with an addition operation $+$ and a multiplication operation $\cdot$ satisfying the requirements that the pair $(S,+)$ is a commutative monoid with identity element $0_S$, the pair $(S,\cdot)$ is a commutative monoid with identity element $1_S$, multiplication distributes over addition and for all $s \in S$, we have $s0 = 0_S$. We write $0$ and $1$ for $0_S$ and $1_S$ if there is no risk of ambiguity.
	\end{definition}
	
	\begin{definition}
		If $S$ is a semiring. An $S$-\emph{semimodule} is a monoid $(M,+)$ with additive identity $0_M$ for which there is a map $S \times M \to M$ sending a pair $(s,m)$ to $sm$ which satisfies the following conditions for all $s_1, s_2 \in S$ and $m_1, m_2 \in M$. We have $(s_1 s_2)m = s_1(s_2 m)$, $s(m_1 + m_2) = sm_1 + sm_2$, $(s_1 + s_2)m = s_1m + s_2m$, $1_S m = m$ and $s 0_M = 0_M = 0_S m$.
	\end{definition}
	
	A semiring $S$, respectively a semimodule $M$ over such semiring, is said to be \emph{additively-idempotent} (or \emph{idempotent} for short) if the additive monoid $(S,+)$, respectively $(M,+)$ is idempotent. Note that if the semiring $S$ is idempotent, then the $S$-semimodule $M$ is idempotent as well. A semiring $S$ is simple if for all $x \in S$ we have $x + 1_S = 1_S$.
	
	\begin{definition}
		If $S_1$ and $S_2$ are semirings, then a map $f: S_1 \to S_2$ is said to be a \emph{semiring morphism} if $f(0_{S_1}) = 0_{S_2}$, $f(1_{S_1}) = 1_{S_2}$ and for all $s_1, s_2 \in S_1$ we have $f(s_1 + s_2) = f(s_1) + f(s_2)$ and $f(s_1 s_2) = f(s_1) f(s_2)$.
	\end{definition}
	
	\begin{definition}
		If $S$ is a semiring and $M_1$ and $M_2$ are $S$-semimodules, then an $S$-semimodule morphism is a map $f: M_1 \to M_2$ such that for any $s \in S$ and any $m_1, m_2 \in M_1$ and $s \in S$ we have $f(m_1 + m_2) = f(m_1) + f(m_2)$ and $f(s m_1) = sf(m_1)$.
	\end{definition}
	
	\begin{definition}
		Let $S$ be a semiring.
		\begin{enumerate}
			\item[(1)] An \emph{ideal} of $S$ is an additive submonoid $I$ such that $SI \subseteq I$. Such ideal is \emph{proper} if it is not $S$.
			\item[(2)] An ideal $I$ of $S$ is said to be \emph{prime} if it is proper and for all $a, b \in S$ if $ab \in I$, then $a \in I$ or $b \in I$.
			\item[(3)] An ideal $I$ of $S$ is said to be a \emph{k-ideal (or substractive ideal)}  if for any $a, b \in S$ if $a, a + b \in I$, then $b \in I$. A \emph{prime $k$-ideal} is a $k$-ideal 
			which is prime as an ideal.
			\item[(4)] A $k$-ideal $I$ is \textit{primary} if given $a,b \in S$ satisfying $ab \in I$ and $a \notin I$, there is an integer $n \geq 1$ such that $b^n \in I$.
			\item[(5)] If $X$ is a subset of $S$ then its \emph{k-closure} is the intersection of of all $k$-ideals of $S$ which contain $X$.
		\end{enumerate}
	\end{definition}
	
	\begin{lemma}\label{Subtractive_Closure_Idempotent_Case}
		If $S$ is idempotent and $X$ is a subset of $S$, then $\langle X \rangle_k = \{ c \in S | a + c = a \text{ for some } a \in \langle X \rangle \}$.
	\end{lemma}
	\begin{proof}
		From the proof of \cite{Lorscheid_Notes}*{p.20, Corollary 2.5.4}, we can write $\langle X \rangle_k$ as the set of $c \in S$ such that $a + c = b$ for some $a,b$ in $\langle X \rangle$. Since $S$ is idempotent, $a + c = b$ is equivalent to $a + c = a$.
	\end{proof}
	
	\begin{definition}\label{Substractive_Module_and_Substractive_Closure}
		Let $N$ be a semimodule over a semiring $S$.
		\begin{enumerate}
			\item[(i)] A nonempty subset $N'$ of  $N$ is said to be a \emph{k-subsemimodule} if each pair $(m, m') \in N \times N$ satisfying $m + m' \in N'$ and $m \in N'$ also satisfies $m' \in N'$.
			\item[(ii)] The \emph{k-closure} of an $S$-semimodule $N'$ of $N$, denoted $\langle N' \rangle_k$, is the intersection of all $k$-subsemimodules of $N$ containing $N'$.
		\end{enumerate}
	\end{definition}
	One has $\langle N' \rangle_k = \{ m \in N | m + n \in N' \text{ for some } n \in N'\}$ (see \cite{Golan_1999}*{p.155}). We now summarize some basic notions of order theory.
	\begin{definition}
		Let $(L,\leq)$ be a poset. We say that $L$ is a \emph{lattice} if for every $x, y \in L$, their greatest lower bound $x \wedge y$ and their lowest upper bound $x \vee y$ are elements of $L$. Such a lattice is said to be
		\begin{enumerate}
			\item[(1)] \emph{bounded} if there are elements $0, 1 \in L$ satisfying $a \wedge 0 = 0$ and $a \vee 1 = 1$ for all $a \in L$.
			\item[(2)] \emph{complete} if for every $X \subseteq L$, $\wedge_{x \in X} x$ and $\vee_{x \in X} x$ exist.
			\item[(3)] \emph{multiplicative} if it has a binary multiplicative operation satisfying $a(b \vee c) = ab \vee ac$ and $(a \vee b)c = ac \vee bc$ for all $a, b, c \in L$.
			\item[(4)] \emph{modular} if for any $a, b, c \in L$ if $a \leq b$ then $a \vee ( c \wedge b) = (a \vee c) \wedge b$.
		\end{enumerate}
	\end{definition}
	\begin{definition}
		A \emph{complete lattice-ordered semigroup} is a multiplicative lattice $L$ which is
		\begin{enumerate}
			\item[(i)] \emph{commutative}, that is, for all $x, \in L$ we have $xy = yx$,
			\item[(ii)] \emph{associative}, that is, for all $x, y, z \in L$ we have $x(yz) = (xy)z$ and
			\item[(iii)] \emph{complete}.
		\end{enumerate}
	\end{definition}
	\begin{definition}
		In a complete lattice $L$, an element $c \in L$ is said to be \emph{compact} if for each set $\Lambda$, for each $\lambda \in \Lambda$, an element $a_\lambda \in L$ satisfying $c \leq \bigvee_{\lambda \in \Lambda} a_\lambda$, there exists a finite subset $\Omega$ of $\Lambda$ such that $c \leq \bigvee_{\lambda \in \Lambda} a_\lambda$. A complete lattice $L$ is \emph{algebraic} if each of its elements is a least upper bound of compact elements.
	\end{definition}
	\begin{definition}
		A semiring $S$ is \emph{lattice-ordered} if it also has the structure of a lattice such that, for all $a, b \in S$: (1) $a+b=a\vee b$, and (2) $ab\leq a\wedge b$,and the partial order $\leq$ is induced by the lattice structure on $S$.
	\end{definition}
	\begin{definition}
		Let $S$ be a lattice-ordered semiring that is complete for the lattice structure. Let $S^c$ be the subset of compact elements in $S$. If $S^c$ is closed under products and contains the multiplicative identity element $1_S$ of $S$, then $S^c$ is a submonoid of $S$ for this structure.
	\end{definition}
	Lattice-ordered semiring morphisms are semiring morphisms (see \cite{Golan_1999}*{p.239, Section 21}).
	\begin{definition}
		We say that a complete lattice-ordered semiring $S$ is an \textit{algebraic lattice} if every element of $S$ is the join of (possibly infinitely many) compact elements.
	\end{definition}
	
	\subsection{Relevant Examples}
	
	The following examples of the theory of Subsection \ref{Subsection_Semistructures_and_Order_Theory} will be needed later.
	
	\begin{example}\label{Example_N}
		On $\N$ we define a lattice structure by setting, for all $a, b \in \N$, $a \wedge b := \mathrm{lcm}(a,b)$ and $a \vee b := \gcd(a,b)$. In this lattice, we have $a \leq b$ precisely when $b$ divides $a$. If $\cdot$ denotes the integer multiplication on $\N$, then $\N^{\gcd} := (\N, \vee, \cdot)$ is a lattice-ordered semiring (see \cite{Golan_1999}*{21.9}).    
	\end{example}
	
	\begin{example}\label{ModMR_Lattice}
		If $R$ is a commutative ring and $M$ is an $R$-module, we let $\Mod(M, R)$ be the set of all $R$-submodules of $M$. The set $\Mod(M,  R)$ can be endowed with a structure of commutative monoid as it is closed under finite sums of $R$-submodules and its identity element is the zero $R$-module. It can also be seen as a poset by defining the following partial order $\leq$. If $N_1, N_2 \in \Mod(M,R)$, then we write $N_1 \leq N_2$ if $N_1 + N_2 = N_2$. The pair $(\Mod(M, R), \leq )$ is also a lattice if we define for any $N_1, N_2 \in \Mod(M, R)$ their greatest lower bound $N_1 \vee N_2$ and their least upper bound to be respectively the $R$-submodules $N_1 \cap N_2$ and $N_1 + N_2$ of $M$. This lattice is bounded because $0, M \in \Mod(M,R)$ are such that for any $N \in \Mod(M,R)$, $N \vee M = N + M = M$ and $N \wedge 0 = N \cap 0 = 0$. Defining the greatest lower bound and the least upper bound on an arbitrary family $\{ N_\lambda \in \id(R) \mid \lambda \in \Lambda \}$ of elements of $\Mod(M, R)$ as the $R$-submodules
		$$
		\bigvee_{\lambda \in \Lambda} N_\lambda := \sum_{\lambda \in \Lambda} N_\lambda \text{ and } \bigwedge_{\lambda \in \Lambda} N_\lambda := \bigcap_{\lambda \in \Lambda} N_\lambda,
		$$
		implies that the lattice $(\Mod(M,R), \leq)$ is complete. Finally, this complete lattice is algebraic because every $N \in \Mod(M, R)$ can be written as the sum of cyclic $R$-submodules $N = \sum_{n \in N} R \cdot n$.
		
		Its subset $\Mod(M, R)^c$ of compact elements is the set of finitely generated $R$-submodules of $M$. In particular, $\Mod(M, R) = \Mod(M, R)^c$ precisely when $R$ is Noetherian. In general, the set $\Mod(M, R)^c$ has a structure of bounded lattice coming from the one of $\Mod(M, R)$. However, if $R$ is not Noetherian, then there is an $N \in \Mod(M, R)$ not in $\Mod(M, R)^c$ with an infinite set $\{ n_\lambda \in N \mid \lambda \in \Lambda \}$ of generators. So $\Mod(M,R)^c$ is not complete, since $N = \sum_{\lambda \in \Lambda} R \cdot n_\lambda = \bigvee_{\lambda \in \Lambda} R \cdot n_\lambda$.
	\end{example}
	
	\begin{example}
		Specializing Example \ref{ModMR_Lattice} to $M = R$, we write $\id(R) := \Mod(R,R)$ and $\id(R)^c := \Mod(R,R)^c$. Dedekind was the first to observe that the set $\id(R)$ together with addition and multiplication of ideals formed an additively-idempotent semiring (see \cite{Golan_1999}*{p. 14, Example 1.4}). This semiring is even simple since for any $I \in \id(R)$ we have $I + R = R$.
	\end{example}
	
	\begin{example}\label{Example_Bezout}
		If $R$ is a B\'ezout domain, and $R^\times$ is its group of units, then $\id(R)^c$ coincides with the quotient $R /R^\times$ and is a multiplicatively cancellative idempotent semiring. As $Frac(\id(R)^c)=Frac(R)/R^\times$, $Frac(\id(R)^c)^\times$ is the \textit{divisibility group} of $R$ of \cite{Rump_Yang_2008}*{p. 1}.
	\end{example}
	
	\begin{example}
		An integral domain $R$ is {\em Pr\"ufer} if every non-zero finitely generated ideal is invertible. Noetherian Pr\"ufer domains are precisely the Dedekind domains, while Pr\"ufer domains which are also GCD domains are the B\'ezout domains. The set $\id(R)^c$ is a marked monoid with no nonzero zerodivisors. Hence, the semiring $\id(R)^c$  has no nonzero zerodivisors. Moreover, $\id(R)^c$ is multiplicatively cancellative if and only $R$ is a Pr\"ufer domain (see \cite{Golan_1999}*{Example 4.36}).
	\end{example}
	
	\begin{example}\label{Fake_Module}
		The action of a ring $R$ on an $R$-module $M$ yields an action of $\id(R)$ on $\Mod(M,R)$, endowing $\Mod(M,R)$ with a structure of $\id(R)$-semimodule. Note that for all $I, J \in \id(R)$, all $N \in \Mod(M,R)$ and all families $\{ J_\lambda \in \id(R) \mid \lambda \in \Lambda \}$ and $\{ N_\lambda \in \Mod(M,R) \mid \lambda \in \Lambda \}$, we have $\left( \sum_{\lambda \in \Lambda} J_\lambda \right)N = \sum_{\lambda \in \Lambda} J_\lambda N$, $J\left( \sum_{\lambda \in \Lambda} N_\lambda \right) = \sum_{\lambda \in \Lambda} J N_\lambda$, $(IJ)N = I(JN)$, $R N = N$ and $0_R N = 0_N$. In the 1970's, D. D. Anderson  studied these objects and called them \emph{fake rings} and \emph{fake modules} in \cite{Anderson_1977}. In particular, if $J \in \id(R)^c$ and $N \in \Mod(M,R)^c$, then $JN \in \Mod(M,R)^c$. With this multiplication, $\Mod(M,R)^c$ is an $\id(R)^c$-semimodule.
	\end{example}
	
	There is a corresponding discussion with $k$-subsemimodules and $k$-ideals.
	
	\begin{example}
		Let $\mathrm{SMod}_k(\Mod(M,R))$ be the set of $k$-subgroups of $\Mod(M,R)$. Note that each of these elements has a structure of $k$-$\id(R)$-subsemimodule of $\Mod(M,R)$, but our results will not use this. Endow with similar operations as in Example \ref{ModMR_Lattice}, by adapting the arguments of \cite{Jun_Ray_Tolliver_2022}*{p.332, Example 3.19} one can see that $\mathrm{SMod}_k(\Mod(M,R))$ is an algebraic lattice and that $\mathrm{SMod}_k(\Mod(M, R))^c$ is the set of finitely generated $k$-subsemigroups of $\Mod(M, R)$. The last conclusion can be reworded by saying that $\SMod_k(\Mod(M,R))^c$ is the set $\SMod_k(\Mod(M,R)^c)$. Once again, one can specialize to the case $M = R$ and write $\id_k(\id(R))$ for $\SMod_k(\Mod(R,R))$. For any two $I,J\in \id_k(\id(R))$, we define their product $I\times J$ as the $k$-closure $\langle IJ \rangle_k$ of the product ideal $IJ$. This gives $\id_k(\id(R))$ a structure of semiring and in fact of lattice-ordered semiring.
	\end{example}
	
	\subsection{Non-Archimedean Seminorms}\label{subsection_NA_Seminorms}
	
	In this subsection, we define non-Archimedean seminorms over Abelian groups, rings and modules over normed rings and give some examples. As before, we write $a \leq b$ to mean that $a + b = b$. Non-Archimedean seminorms on an Abelian group are defined as follows.
	
	\begin{definition}
		\label{def_gen_nav_monoids}
		A {\em non-Archimedean seminorm} is a map $|\cdot|:A\xrightarrow{}{N}$ from an Abelian group $(A,+,0)$ to an idempotent monoid ${(N,+,0)}$ that satisfies
		\begin{enumerate}
			\item (unit) $|0|=0$ \text{ and }
			\item (subadditivity) $|a+b|\leq |a|+|b|$ for all $a, b \in A$.
		\end{enumerate}
		Such non-Archimedean seminorm $|\cdot|$ is a {\em norm} if each $a \in A \smallsetminus \{0\}$ satisfies $|a| \ne 0$ and we call the triple $(A, | \cdot |, {N})$ a \textit{seminormed group} (respectively, \textit{normed group}).
	\end{definition}
	We now turn to rings. Inspired by the concept of non-Archimedean seminorm on a ring as in \cite{Berkovich_1990}*{p. 11}, we ask a non-Archimedean seminorm to be submultiplicative, which generalizes the multiplicative valuation of \cite{Giansiracusa_Giansiracusa_2016}*{p.3989, Definition 2.5.1}. If $S$ is totally ordered, the latter is the definition of a \emph{Krull valuation}.
	
	\begin{definition}
		\label{def_gen_nav}
		A {\em non-Archimedean seminorm} is a map $|\cdot|:R\xrightarrow{}S$ from a ring $R$ to an idempotent semiring $S$ such that
		\begin{enumerate}
			\item (unit) $|0|=0$ and $|1|=1$,
			\item (sign) $|-1|=1$,
			\item (submultiplicativity) $|ab|\leq |a||b|$ for all $a, b \in R$ and
			\item (subadditivity) $|a+b|\leq |a|+|b|$ for all $a, b \in R$. 
		\end{enumerate}
		The non-Archimedean seminorm $|\cdot|$ is {\it multiplicative} if $|ab|= |a||b|$ holds for all $a,b \in R$, and it is a {\em norm} if each non-zero $a \in R$ we have $|a| \ne 0$. We keep the term {\em valuation} for a multiplicative norm. We call the triple $(R, | \cdot |, S)$ a \textit{seminormed ring} (respectively, \textit{normed ring}). Since only non-Archimedean seminorms are considered, we call them \textit{seminorms} and denote them using lowercase Roman letters instead of bars $| \cdot |$. 
	\end{definition}
	
	\begin{example}\label{example_special_cases}
		Let $v:R\to S$ be a multiplicative seminorm. If $S$ is an idempotent semifield, one can construct $\ell$-valuations (see \cite{Rump_Yang_2008}*{p. 266}). In particular, we have $v(-1) = 1$ since $(S\setminus\{0\},\times,1)$ is torsion-free. More simply, if $R$ is a field, then $v$ is a valuation.
	\end{example}
	\begin{definition}
		\label{def:int_seminorm}
		A seminorm $v:R\xrightarrow{}S$ is {\em integral} if $v(a) \leq 1$ for all $a \in R$.
	\end{definition}
	\begin{remark}\label{Remark_Integral_Seminorm}
		If $S$ is an idempotent semiring, then the set $S_{\leq 1} := \{y\in S : y\leq1\}$ is a simple semiring and we let $i: S_{\leq 1} \to S$ be the inclusion map. If $v:R\xrightarrow{}S$ seminormed ring, then the subring $v^{-1}(S_{\leq 1})$ of $R$ is called the \emph{seminorm ring} of $v$.
		Note that $v$ is integral if and only if it factors as $i \circ v'$, where $v':R\to S_{\leq 1}$ is a seminorm. From now on, if  $v$ is an integral seminorm, then we assume that $S$ is a simple semiring. In particular $b a\leq a$ for all $a,b\in S$ as we have $b a+a=a(b+1)=a$.
	\end{remark}
	Following \cite{Berkovich_1990}*{p. 12}, we consider norms defined on modules over normed rings.
	\begin{definition}
		\label{def_normod}
		Let $v: R \to S$ be a normed ring, let $M$ be an $R$-module and let $N$ be an $S$-semimodule. A seminorm of Abelian groups $w:M\xrightarrow{}N$ is a $v$-{\em seminorm} if there is a nonzero $C \in S$ such that
		\begin{equation}
			\label{eq_inequality}
			w(r\cdot m)\leq Cv(r)\cdot w(m)
		\end{equation}
		for each $r \in R$ and $m \in M$. The $v$-seminorm $w$ is called a $v$-{\em norm} if each $m \in M \smallsetminus \{0\}$ satisfies $w(m) \ne 0$.
	\end{definition}
	
	\section{Universal Valuation and Universal Norm}\label{Section_Universal}
	
	\subsection{The Universal Valuation}\label{Subsection_Universal_Valuation}
	If $f: R_1 \xrightarrow{} R_2$ is a ring morphism, we let $\id^c(f)$ be the map that sends an $I \in \id^c(R_1)$ to $ \langle f(I) \rangle \in \id^c(R_2)$. Since each $\id(R)^c$ is an idempotent (even simple) semiring for the sum and product of ideals and as a result each $\id^c(f)$ is a semiring morphism, this yields a functor $\id^c$ from the category $\Ring$ of commutative rings to the category $\SRing^{\text{idem}}$ of idempotent semirings.
	
	The so-called \textit{universal valuation of $R$} is the map $u_R: R\xrightarrow{}\id(R)^c$ that sends an element $a \in R$ to the principal ideal $\langle a \rangle:= R\cdot a$. Given a ring morphism $f: R_1 \to R_2$ one has the commutative diagram $\id^c(f) \circ u_{R_1} = u_{R_2} \circ f$. As observed in \cite{Giansiracusa_Giansiracusa_2016}*{p.3391, Proposition 2.5.4} and \cite{Macpherson_2020}*{p.449, Example 3.16}, the map $u_R$ is an integral seminorm and it is moreover universal among the seminorms defined over $R$. More precisely, any seminorm that takes values in an idempotent semiring, or equivalently simple semiring because of Remark \ref{Remark_Integral_Seminorm}, factors uniquely through $u_R$. For a ring $R$, there is a set-functor $\mathrm{isn}_R$ from the category of simple semirings, that sends a simple semiring $S$ to the set $\mathrm{isn}_R(S)$ of integral seminorms $v: R \to S$. In the above two \textit{loc. cit.}, the authors also deduce that $\mathrm{ins}_R$ is represented by $\id(R)^c$.
	
	\begin{remark}\label{Remark_choose_generators}
		Let $f: R \to S$ be a ring morphism. Let $J \in \id(R)^c$ with generating set $\{a_1,\cdots,a_m\}$. Then, $\id^c(f)(\sum_{i=1}^m Ra_i)=\sum_{i=1}^m Sf(a_i)$. This map does not depend on the choice of generators for $J$. Indeed, if $\{c_1, \cdots, c_n\}$ is a second set of generators, then for each $1 \leq j \leq n$ there are $x_{1,j}, \cdots, x_{m,j} \in R$ such that $c_j = \sum_{i=1}^m x_{i,j} a_j$. Then, $f(c_j)$ equals $\sum_{i=1}^m f(x_{i,j}) f(a_i)$ and so $\id^c(f)(\sum_{i=1}^m R a_i) = \sum_{j=1}^n S f(c_j)$. By symmetry, $\id^c(f)(\sum_{j=1}^n R c_j) = \sum_{i=1}^m S f(a_i)$.
	\end{remark}
	\begin{lemma}\label{lemma_ind_hom}
		If $v:R\to S$ is an integral seminorm, then the map $\hat v:\id(R)^c\to S$ that sends $J = \sum_{i=1}^m R a_i$, to $\sum_{i=1}^m v(a_i)$ is a semiring morphism.
	\end{lemma}
	\begin{proof}
		The map $\hat{v}$ is well defined. Indeed, if $J=\sum_{i=1}^m R a_i=\sum_{j=1}^n R b_j$, then for each $1 \leq i \leq m$, there exist $c_{i,j} \in R$ with $1 \leq j \leq n$ such that $a_i=\sum_{j=1}^n c_{i,j}b_j$, and thus $v(a_i)\leq \sum_{j=1}^n v(c_{i,j})v(b_j)$. Since $S$ is simple, we have $v(c_{i,j})v(b_j)\leq v(b_j)$ for all $j$, so $v(a_i)\leq \sum_{j=1}^n v(c_{i,j})v(b_j)\leq \sum_{j=1}^n v(b_j)$, hence $\sum_{i=1}^m v(a_i)\leq m \sum_{j=1}^n v(b_j) = \sum_{j=1}^n v(b_j)$. By symmetry, $ \sum_{j=1}^n v(b_j)\leq \sum_{i = 1}^m v(a_i)$, and so $\hat{v}$ is well defined.
		
		The map $\hat{v}$ is a semiring morphism. Indeed, it follows immediately from the definition of $\hat{v}$ that $\hat v(0)= \langle 0 \rangle$ and $\hat v(1)= \langle 1 \rangle$. Finally, if $I=\sum_{i=1}^m Ra_i$ and $J=\sum_{j=1}^nR b_j$, then $\hat v(\alpha+\beta)=\sum_{i=1}^m v(a_i) + \sum_{j=1}^n v(b_j) = \hat v(\alpha)+\hat v(\beta)$ and $\hat v(\alpha\beta)=\sum_{j=1}^n\sum_{i=1}^m v(a_i)v(b_j)= \hat{v}(\alpha) \hat{v}(\beta)$.
	\end{proof}
	
	\subsection{The Universal Norm}\label{Sect_Val_Mods}
	In this subsection, we fix a ring $R$ over which all modules are going to be defined. If $M$ is an $R$-module, the set $\Mod(M,R)$ of $R$-submodules of $M$ is an additive idempotent monoid for the sum of $R$-submodules. We write $\Mod(M)$ for $\Mod(M,R)$ and $\Mod(M)^c$ for its submonoid of compact elements when there is no ambiguity. Given an $R$-module morphism $f:M_1\xrightarrow{}M_2$, we denote by $\Mod^c(f)$ the map that sends an $N \in \Mod^c(M_1)$ to $\langle f(N) \rangle = f(N)$ $\Mod^c(M_2)$. So if $N = \sum_{i=1}^r R n_i$, then $f(N) = \sum_{i=1}^r \Mod^c(f)(R n_i)$. As mentioned in Example \ref{Fake_Module}, the action of $R$ on $M$ turns $\Mod(M)$ into an idempotent $\id(R)$-semimodule and $\Mod(M)^c$ into an idempotent $\id(R)^c$-semimodule. The map $\Mod(f)$ then becomes an $\id(R)^c$-semimodule morphism and we thus have a functor
	$$
	\Mod^c : R-\Mod \to \id(R)^c-\mathrm{SMod},
	$$
	to the category $\id(R)^c-\mathrm{SMod}$ of $\id(R)^c$-semimodules. Similarly to Proposition \ref{lemma_ind_hom}, we have the following.
	
	\begin{lemma}
		\label{lemma_u_M_univ}
		Let $w:M\to N$ be an $u_R$-seminorm.
		\begin{enumerate}
			\item[(a)] The map $\hat w$ that sends an element $T=\sum_{i=1}^m R t_i$ of $\Mod(M)^c$ to $\sum_{i=1}^m w(t_i) \in N$ is a monoid morphism satisfying $\hat w(J T)\leq J\hat w(T)$ for all $J \in \id(R)^c$.
			\item[(b)] The map $\hat w$ from (a) is an $\id(R)^c$-semimodule morphism if and only if $w$ satisfies the following stronger version of \eqref{eq_inequality}: For all $r \in R$ and $m \in M$, we have
			\begin{equation}\label{eq_strong_inequality}
				w(rm)=u_R(r)w(m).
			\end{equation}
		\end{enumerate}
	\end{lemma}
	\begin{proof}
		\begin{enumerate}
			\item[(a)] 
			The map $\hat w$ is independent of the choice of generators for $T \in \Mod(M)^c$. Indeed, if $T$ has generating sets $\{t_1,\cdots, t_m\}$ and $\{s_1,\cdots,s_n\}$, then each $s_j$ can be written as a sum $s_j=\sum_{i=1}^mc_{j,i}  t_i$ for some $c_{j,i} \in R$. Thus, there is a $C\in \id(R)^c \smallsetminus \{0\}$ such that
			\begin{equation}\label{Equation_Norm}
				w(s_j)\leq \sum_{i=1}^m w(c_{j,i} t_i)\leq \sum_{i=1}^m Cu_R(c_{j,i}) w( t_i).
			\end{equation}
			Now, $Cu_R(c_{i,j}) \leq R$ in $\id(R)^c$ and together with \eqref{Equation_Norm} this yield $\sum_{j=1}^n w(s_j) \leq \sum_{i=1}^m w(t_i)$. Similarly, we obtain $ \sum_{i=1}^w(t_i)\leq \sum_{j=1}^n w(s_j)$.
			
			The map $\hat{w}$ is a monoid morphism. Indeed, we have $\hat w(0)=(0)$ and for $T_a = \sum_{i=1}^m a_i$ and $T_b = \sum_{j=1}^n b_j$ in $\Mod(M)^c$, we have $\hat w(T_a + T_b)=\sum_{i=1}^m w(a_i)+\sum_{j=1}^n w(b_j)=\hat w(T_a)+\hat w(T_b)$.
			
			Finally, for $J = \sum_{j=1}^n R x_j \in \id(R)^c$ and $T = \sum_{i=1}^m R t_i \in \Mod(M)^c$ we have $\hat w( J T) = \sum_{i,j} w( x_j t_i )$ and $J \hat w(T) = \sum_{i,j} Rx_j w(t_i)$. Thus,
			$$
			\hat w(J T)=\sum_{i,j}w(x_j t_i)\leq \sum_{i,j}C R x_j w(t_i)\leq \sum_{i,j}R x_j w(t_i)=J\hat w(T),
			$$
			which concludes the proof.
			\item[(b)] If \eqref{eq_strong_inequality} holds, then $\hat w(JT)=\hat w(\sum_{i,j}(x_j)(t_i))=\sum_{i,j}\hat w(x_j t_i)=J\hat w(T)$, and in particular $J\hat w(T)\leq \hat w(J T)$. Conversely, we recover \eqref{eq_strong_inequality} by applying $\hat w(JT)= J\hat w(T)$ to $J=Rr$ and $T=Rm$.\qedhere
		\end{enumerate}
	\end{proof}
	
	Next, we prove a universal seminorm analogue of \cite{Giansiracusa_Giansiracusa_2016}*{Proposition 2.5.4} and \cite{Macpherson_2020}*{p.449, Example 3.16} for the $u_R$-norm $u_M : M \to \Mod(M,R)^c$ (see Subsection \ref{Subsection_Overview}).
	
	\begin{proposition}\label{prop_snorm_module}
		The  map $u_M$ satisfies the following properties:\begin{enumerate}
			\item[(i)] Any $u_R$-seminorm on $M$ factors uniquely through $u_M$.
			\item[(ii)] If $f:M_1\xrightarrow{}M_2$ is an $R$-module morphism, then $u_{M_2} \circ f = \Mod^c(f) \circ u_{M_1}$.
			\item[(iii)] For an $R$-module $M$, the functor $\text{sn}_{u_R}$ from the category of $\id(R)^c$-semimodules to the category of sets that sends a $\id(R)^c$-semimodule $N$ to the set  $\text{sn}_{u_R}(N)=\{u_R-\text{seminorms } w: M \to N\}$ is represented by $\Mod(M,R)^c$.
		\end{enumerate}
	\end{proposition}
	
	\begin{proof}
		\begin{enumerate}
			\item[(i)] It follows from Lemma \ref{lemma_u_M_univ} (b) and then (a) that $u_M$ is the \emph{universal $u_R$-norm} and then that any $u_R$-seminorm factors uniquely through it.
			\item[(ii)] Observe that $R(a+b)=u_M(a+b)\leq u_M(a)+u_M(b)=R a+R b$ holds. Also, $R a = \langle 0 \rangle$ if and only if $a=0$ since $1\cdot a=a$. As  $u_M(r m)=\langle r m\rangle=u_R(r) u_M(m)$, we can pick $C= \langle 1 \rangle$ to verify Definition \ref{def_normod}. That the diagram commutes is verified straightforwardly.
			\item[(iii)] Since $u_M$ satisfies \eqref{eq_strong_inequality} and $\text{sn}_{u_R}(N)=\operatorname{Hom}_{\mathrm{SMod}}(\Mod(M),N)$, then $\Mod(M)^c$ represents $\text{sn}_{u_R}$.
		\end{enumerate}
	\end{proof}
	
	\subsection{Module and Ring Correspondences} 
	\label{sect:properties}
	
	The next result is well-known and follows for example from \cite{Jun_Ray_Tolliver_2022}*{Theorems 3.24 and 3.29}.
	\begin{proposition}[Correspondence for Submodules]\label{Correspondence_Modules_over_a_Ring}
		The $u_R$-norm $u_M : M \to \Mod(M,R)^c$ defines an isomorphism of ordered sets: 
		\[
		\begin{tikzcd}
			\mathrm{Mod}(M,R) \arrow[rr, "\langle u_M(\cdot)\rangle_k", shift left=3] &  & \mathrm{SMod}_k\left( \Mod(M,R)^c \right) \arrow[ll, "u_M^{-1}(\cdot)", shift left=3].
		\end{tikzcd}
		\]
		In particular, the two maps induced by the integral valuation $u_R : R \to \id(R)^c$ are both ordered set and semiring isomorphisms
		\begin{equation}\label{Equation_Ordered_Set_and_Semiring_Isomorphism}
			\begin{tikzcd}
				(\mathrm{Id}(R),\subseteq) \arrow[rr, "\langle u_R(\cdot)\rangle_k", shift left=3] &  & (\mathrm{Id}_k\left( \id(R)^c \right),\subseteq). \arrow[ll, "u_R^{-1}(\cdot)", shift left=3]
			\end{tikzcd}
		\end{equation}
	\end{proposition}
	
	We now consider the effect of the universal valuation on primary ideals and hence on prime ideals.
	
	\begin{proposition}\label{Correspondence_Ideals_Preserves_Primes}
		The correspondence \eqref{Equation_Ordered_Set_and_Semiring_Isomorphism} preserves primary ideals.
	\end{proposition}
	\begin{proof}
		Let $\mathfrak{q} \subset \id(R)^c$ be a primary $k$-ideal. Let $a, b \in R$ such that $ab^n \in u_R^{-1}(\mathfrak{q})$ for some integer $n \geq 1$. Since $u_R(a)u_R(b^n) = u_R(ab^n)$ and $\mathfrak{q}$ is primary, then $u_R(a) \in \mathfrak{q}$ or $u_R(b^n) \in \mathfrak{q}$. So $a \in u_R^{-1}(\mathfrak{q})$ or $b^n \in u_R^{-1}(\mathfrak{q})$, and $u_R^{-1}(\mathfrak{q})$ is primary. Conversely, let $\mathfrak{q} \subset R$ be a primary ideal. Let $\alpha, \beta \in \id(R)^c$ be such that $\alpha \beta^n \in \langle u_R(\mathfrak{q}) \rangle_k$ for some integer $n \geq 1$. Since $\id(R)^c$ is idempotent, there is a $\gamma \in \langle u_R(\mathfrak{q}) \rangle$ such that $\gamma + \alpha = \gamma$. Since $\mathfrak{q}$ is an ideal, then $\alpha \beta^n \subset \mathfrak{q}$. But $\mathfrak{q}$ is primary and so $\alpha \subset \mathfrak{q}$ or $\beta^n \subset \mathfrak{q}$. Writing $\alpha$ and $\beta^n$ in terms of respective finite generating sets. we deduce that either $\alpha \in \langle u_R(\mathfrak{q})\rangle \subset \langle u_R(\mathfrak{q}) \rangle_k$ or $\beta^n \in \langle u_R(\mathfrak{q}) \rangle \subset \langle u_R(\mathfrak{q}) \rangle_k$, concluding.
	\end{proof}
	\begin{definition}\label{Definition_Affine_Schemes}
		The set of prime $k$-ideals of $\id(R)^c$ is denoted $\Spec_k(\id(R)^c)$ and its subset of prime $k$-ideals containing a set $X$ is denoted $V_k(X)$.
	\end{definition}
	\begin{lemma}\label{Lemma_Closed_Subset_V_and_Vk}
		For each subset $X$ of $\Spec_k(\id(R)^c)$, we have $V_k(X) = V_k(\langle X \rangle) = V_k(\langle X \rangle_k)$.
	\end{lemma}
	\begin{proof}
		One immediately verifies that $V_k(X) = V_k(\langle X \rangle)$, and $ \langle X \rangle \subset \langle X \rangle_k$ yields $ V_k(\langle X \rangle) \supset V_k(\langle X \rangle_k)$. Now for each $\pp \in \Spec_k(\id(R)^c)$ containing $\langle X \rangle$, we have
		$\langle X \rangle_k \subset \langle \pp \rangle_k = \pp$ and thus, $V_k( \langle X \rangle) = V_k( \langle X \rangle_k)$.
	\end{proof}
	\begin{definition}\label{Definition_Distinguished_Open_Subsets}
		For each $k$-ideal $I$ of $\id(R)^c$, we denote $D_k(I)$ the complement of $V_k(I)$ in $\Spec_k(\id(R)^c)$. In particular, if $f \in \id(R)^c$, we write $D_k(f)$ for $D_k(\langle f \rangle_k) = \{ \pp \in \Spec_k(\id(R)^c) : f \notin \pp \}$.
	\end{definition}
	By Proposition \ref{Correspondence_Ideals_Preserves_Primes}, the universal valuation $u_R : R \to \id(R)^c$ induces a set bijection 
	\begin{equation}
		\label{Homeomorphism_SpecR_SpeckfgIdR}
		\widetilde{u_R}:\Spec(R)\to \Spec_k(\id(R)^c).
	\end{equation}
	\begin{proposition}
		\label{prop_homeo_real}
		The set $\Spec_k(\id(R)^c)$ is a topological space for the closed subsets $V_k(X)$, as $X$ runs over the subsets of $\id(R)^c$. The topology coincides with the one induced by \eqref{Homeomorphism_SpecR_SpeckfgIdR}. Hence, $\widetilde{u_R}$ is a homeomorphism. 
	\end{proposition}
	\begin{proof} From Lemma \ref{Lemma_Closed_Subset_V_and_Vk}, it suffices to consider $k$-ideals. Since $V_k(\id(R)^c) = \emptyset$, then $D_k(\id(R)^c) = \Spec_k(\id(R)^c)$. Let $I_1, I_2 \in \id(R)$. From Proposition \ref{Correspondence_Ideals_Preserves_Primes} we have
		$$
		V(I_1) \cup V(I_2) = V(I_1 I_2) = V_k\left( \langle u_R(I_1) u_R(I_2) \rangle_k \right) = V_k(u_R(I_1)) \cup V_k( u_R(I_2)).
		$$
		By induction, for any integer $n \geq 1$ if $I_1, \cdots, I_n \in \id(R)$, then $\cup_{r=1}^n V(I_r) = \cup_{r=1}^n V_k\left( u_R(I_r)\right)$.
		Finally, we finish the proof by seeing that for any $\Lambda$ and any family $\{ I_\lambda | \lambda \in \Lambda \}$ of $k$-ideals of $\id(R)^c$, we have
		$$\bigcap_{\lambda \in \Lambda} V_k(I_\lambda) = V_k\left( \sum_{\lambda \in \Lambda} I_\lambda\right) = V\left( \sum_{\lambda \in \Lambda} u_R^{-1}\left(I_\lambda\right)\right)=\bigcap_{\lambda \in \Lambda} V(u_R^{-1}\left(I_\lambda\right)).
		$$
	\end{proof}
	\begin{definition}\label{def_topo_space}
		The topology on the set $Spec_k(\id(R)^c)$ from Proposition \ref{prop_homeo_real} is called the \textit{Zariski topology}.
	\end{definition}
	\begin{remark}
		The topological space $Spec_k(\id(R)^c)$ is a \textit{spectral space}, while the set $Spec(\id(R)^c)$ endowed with the Zariski topology of all prime ideals of $\id(R)^c$ is not. The inclusion map $i:\Spec_k(\id(R)^c)\to Spec(\id(R)^c)$ is continuous since if $I \in Id(\id(R)^c)$, then $i^{-1}(V(I))=V_k(I)$. Note also that $\Spec_k(\id(R)^c)$ is irreducible if and only if it has no nonzero zerodivisors (cf.  \cite{Jun_2015}*{p.51, Lemma 2.2.7}).
	\end{remark}
	
	For the basic definitions and properties of localization $V^{-1}S$ of a semiring $S$ at a multiplicatively closed subset $V\subset S$, we refer the reader to \cite{Jun_2015}*{p. 12}. Note that if $S$ is idempotent, then so is $V^{-1}S$. 
	
	\begin{proposition}\label{prop_loc_val}
		If $U \subseteq \id(R)^c$ is a multiplicatively closed subset, then so is $W := u_R^{-1}(U) \subseteq R$. If $U$ has no nonzero zerodivisors, then the map $\mu:W^{-1}R\to U^{-1}\left(\id(R)^c\right)$ defined by $a/b \mapsto u_R(a)/u_R(b)$ is a seminorm and we have a commutative diagram:
		\begin{equation}
			\begin{tikzcd}
				W^{-1}R \arrow[dd, "\mu"', bend right=67] \arrow[d, "u_{W^{-1}R}"] &  &                                                       & R \arrow[lll, "\iota_W"] \arrow[ld, "u_R"'] \arrow[llldd, "\iota_U", bend left] \\
				\id(W^{-1}R)^c                                             &  & \id(R)^c \arrow[ll, "\id^c(\iota_W)"] &                                                                                 \\
				U^{-1}\left(\id(R)^c\right)\arrow[u, "\exists !~\zeta"', dashed]       &  &                                                       &                                                                                
			\end{tikzcd}
		\end{equation}
		where the morphisms $\iota_U$ and $\iota_W$ are the localization morphisms.
	\end{proposition}
	\begin{proof}
		We have $1 \in W$ and if $a,b\in W$, then $u_R(ab)=u_R(a)u_R(b)\in U$. As $u_R$ is surjective, $u_R(W)=U$. The map $\mu:W^{-1}R\to U^{-1}\left(\id(R)^c\right)$ sending $a/b$ to $u_R(a)/u_R(b)$ is well defined since if $a/b=c/d$, then $u_R(a)u_R(d)=u_R(ad)=u_R(bc)=u_R(b)u_R(c)$, which implies $\mu(a/b)=\mu(c/d)$. That $\mu$ is a seminorm follows from the fact that $u_R: R\to \id(R)^c$ is a valuation. For instance, if $a/b, c/d \in W^{-1}R$, then
		\begin{enumerate}
			\item  $\mu(\frac{a}{b} \frac{c}{d})=\mu(\frac{ac}{bd})=\frac{u_R(ac)}{u_R(bd)}=\mu(\frac{a}{b})\mu(\frac{c}{d})$,
			\item $\mu(\frac{a}{b}+\frac{c}{d})=\mu(\frac{ad+bc}{bd})=\frac{u_R(ad+bc)}{u_R(bd)}\leq \frac{u_R(ad)+u_R(bc)}{u_R(bd)}=\mu(\frac{a}{b})+\mu(\frac{c}{d})$.
		\end{enumerate}
		Note that the seminorm $\mu$ is no longer integral. In fact, $\mu(\frac{a}{b})\leq_{U^{-1}\left(\id(R)^c\right)}1$ if and only if $u_R(a)\leq_{\id(R)^c} u_R(b)$. Let $\alpha \in U$. Since $u_R(W)=U$, we have $u_R(\alpha)=\langle a_1,\ldots,a_m \rangle$ with $a_i\in W$. As $\id^c(\iota_W)(u_R(\alpha))= \langle (\iota_W(a_1),\ldots,\iota_W(a_m) \rangle =\alpha$, each $\iota_W(a_i)$ is a unit in $W^{-1}R$. Thus $\id^c(\iota_W)(u_R(\alpha))=1$ and so $\alpha = 1$. Thus the existence of the map $\zeta: U^{-1}\id(R)^c \to \id(W^{-1}R)^c$ follows from the universal property of the localization for semirings. As in the case of commutative rings, the map $\zeta$ is determined by $\id^c(\iota_W) \circ u_R$ as
		$$
		\zeta\left( \frac{t}{u} \right) = \zeta(t) \zeta\left( \frac{1}{u} \right) = \left( \id^c(\iota_W) \circ u_R \right)(t) \cdot \left( \id^c(\iota_W) \circ u_R \right)(u)^{-1}
		$$
		and is well-defined. Now, there exist unique $\alpha = \langle a_1, \cdots, a_m \rangle$ and $\beta = \langle b_1, \cdots, b_n \rangle$ in $\id(R)^c$ such that $u_R(\alpha) = \sum_{i=1}^m u_R(a_i) = t$ and $u_R(\beta) = \sum_{j=1}^n u_R(b_j) = u$. Moreover, $\beta \in W$ since $u \in U$. As observed earlier, $\id^c(\iota_W)(\beta) = 1$. Therefore,
		$$
		\zeta\left( \frac{t}{u} \right) = \id^c(\iota_W)(\langle a_1, \cdots, a_m \rangle) = \left\langle \frac{a_1}{1}, \cdots, \frac{a_m}{1} \right\rangle.
		$$
		To conclude, we show that $\mu \circ \zeta = u_{W^{-1}R}$. Let $r/w \in W^{-1}R$. We have
		$\left( \zeta \circ \mu \right)\left( r/w \right) = \zeta\left( u_R(r)/u_R(w) \right) = \left\langle r/1 \right\rangle$ and $u_{W^{-1}R}\left( r/w \right) = \left\langle r/w\right\rangle = \left \langle r/1 \right\rangle$, which concludes the proof.
	\end{proof}
	
	\begin{proposition}\label{Localization_Commutes_with_kPrimes_in_fgId}
		Let $U \subseteq \id(R)^c$, $u_R:R\to \id(R)^c$ and $W=v^{-1}(U)\subset R$ as in Proposition \ref{prop_loc_val}. Then we have a set bijections
		$$
		\Spec\left(W^{-1}R \right)\simeq \Spec_k\left( \id\left( W^{-1}R \right)^c \right) \simeq \Spec_k\left( U^{-1}\left(\id(R)^c\right)\right).
		$$
	\end{proposition}
	\begin{proof}
		A point $\mathfrak{q}^\dag$ of $\Spec_k\left(U^{-1}\left(\id(R)^c\right)\right)$ corresponds to a unique prime $k$-ideal $\mathfrak{p}^\dag$ of $\id(R)^c$ such that $\mathfrak{p}^\dag \cap U = \emptyset$ (see \cite{Lorscheid_Notes}*{p.24, Proposition 2.7.5}). Applying Proposition \ref{Correspondence_Ideals_Preserves_Primes} twice, $\mathfrak{p}^\dag$ corresponds to a unique prime ideal $\pp$ of $R$ such that $\pp \cap W = \emptyset$ and then $\pp$ corresponds to a unique point $\pp^\dag \in \Spec_k\left( \id^c(W^{-1}R) \right)$.
	\end{proof}
	
	However, we stress that the map $\zeta : U^{-1}(\id(R)^c) \to \id(W^{-1}R)^c$ does not have to be an isomorphism. Said differently, the functor $\id^c$ does not commute with the localization functor in general.
	\begin{example}
		Since $\Z$ is Noetherian, $\id(\Z)^c = \id(\Z)$. One then sees that $\id(\Z)$ is isomorphic to the semiring $\N^{\gcd}$ from Example \eqref{Example_N}. With the bijection $\widetilde{u_\Z}$ we deduce that the elements of $\Spec_k(\N^{\gcd})$ are the ideal $0$ and, for each prime number $p$, the set $(p\Z)^\dagger = \{ m \in \N \mid p \mid m \}$.
		However,
		$$
		\id^c(\Z_\pp) =
		\begin{cases}
			\id^c(\Q) = \mathbb{B} &\text{if } \pp = 0,\\
			\id^c(\Z_{\langle p \rangle}) \simeq \{ n\Z : p \not | n \}&\text{if } \pp = p\Z,
		\end{cases}
		\text{ and }
		\N_{\pp^\dagger} =
		\begin{cases}
			\Q_{> 0} &\text{if } \pp = 0,\\
			\{ \frac{a}{b} \in \Q_{> 0} : p \not | b \} &\text{if } \pp = p\Z.
		\end{cases}
		$$
	\end{example}
	\begin{remark}\label{rem_homeo}
		It follows from Proposition \ref{Localization_Commutes_with_kPrimes_in_fgId} that for $f \in R$ we have a chain of homeomorphisms:
		\begin{equation}\label{Equation_Identification_Distinguished_Opens}
			Spec(R_f)\cong D(f)\cong D_k(f)\cong  Spec_k(\id(R)^c_{u_R(f)}).
		\end{equation}
		Likewise, if $I\subset R$ is an ideal, then we have a homeomorphism between the closed sets $V(I)\subset Spec(R)$ and $V_k(\langle u_R(I) \rangle_k)\subset Spec_k(\id(R)^c)$. In Corollary \ref{Corollary_of_Correspondence}, we show that \eqref{Equation_Identification_Distinguished_Opens} extends to a chain of homeomorphisms:
		\begin{equation}
			Spec(R/I)\cong V(I)\cong V_k(\langle u_R(I) \rangle_k)\cong  Spec_k(\id(R)^c/\langle u_R(I) \rangle_k).
		\end{equation}
	\end{remark}
	
	\begin{remark}\label{Remark_Preserving_Open_Cover}
		Let $R$ be a ring and let $\Sigma$ be a generating subset of $R$. Then, the set $\calU_\Sigma := \{ D(f) \mid f \in \Sigma \}$ is an affine open covering of $\Spec(R)$. Since $\widetilde{u_R} : \Spec(R) \to \Spec_k(\id(R)^c)$ is a homeomorphism and since by \eqref{Equation_Identification_Distinguished_Opens}, there is a homeomorphism $D(f) \to D_k(f)$ for each $f \in \Sigma$, then $\calU_{\idm(\Sigma)} := \{ D_k(f) \mid f \in \Sigma \}$ is an open covering of $\Spec_k(\id(R)^c)$.
	\end{remark}
	
	We now study the radical of $k$-ideals.
	\begin{definition}\label{definition_sring_radical}
		Let $I\in Id_k(\id(R)^c)$. The \textit{radical} $\sqrt{I}$ of $I$ is the set of $a \in \id(R)^c$ satisfying $a^n \in I$ for some integer $n \geq 1$. We say that $I$ is \textit{radical} if $I = \sqrt{I}$.
	\end{definition}
	\begin{proposition}
		\label{proposition_radical_ideals}
		If $I \in Id_k(\id(R)^c)$, then $\sqrt{I}$ is a $k$-ideal satisfying
		\begin{equation}
			\label{radicals}
			\bigcap_{\pp \in V_k(I)} \pp = \{ a \in \id(R)^c | a^n \in I \text{ for some integer } n \geq 1\}=\sqrt{I}.
		\end{equation}
		Also, the correspondence \eqref{Equation_Ordered_Set_and_Semiring_Isomorphism} preserves radical ideals.
	\end{proposition}
	\begin{proof}
		As in the case of rings, a proper $k$-ideal $I\subset \id(R)^c$ has a radical $\sqrt{I}$ which is also a $k$-ideal and coincides with the intersection of all the primes containing $I$. This follows at once from \cite{Fuchs_Reis_2003}*{Lemma 2.5} which says that $\sqrt{I}$ is the intersection of the minimal $k$-primes of $I$. Thus, \eqref{Equation_Ordered_Set_and_Semiring_Isomorphism} preserves radicals since it is a poset isomorphism. Note that $\sqrt{\id(R)^c} = \id(R)^c$ and $V_k(\id(R)^c) = \emptyset$. Hence $\cap_{\pp \in \emptyset} \pp = \id(R)^c$.
	\end{proof}
	
	\begin{corollary}\label{corollary_vsets_and_radicals}
		If $I, J \in Id_k(\id(R)^c)$, then $V_k(I) \subseteq V_k(J)$ if and only if $\sqrt{I} \supseteq \sqrt{J}$.
	\end{corollary}
	\begin{proof}
		By Proposition \ref{proposition_radical_ideals}, $\sqrt{I} = \bigcap_{\pp \in V_k(I)} \pp$ and $\sqrt{J} = \bigcap_{\pp \in V_k(J)} \pp$ and then the result follows.
	\end{proof}
	\begin{remark}
		Since the definition of {\it primary element} from \cite{Fuchs_Reis_2003}*{p. 346} coincides with Definition \ref{definition_sring_radical}, then the radical of a primary $k$-ideal is a prime $k$-ideal.
	\end{remark}
	
	\subsection{Retraction Theorems and Consequences}
	
	Before proving our retractions theorems and discussing some of their consequences, we introduce needed order theory notions. We follow \cite{Blyth_2005} and \cite{Dickmann_Schwartz_Tressl_2019}.
	
	\begin{definition}
		Let $(P, \leq)$ be an ordered set.
		\begin{enumerate}
			\item[(i)] A subset $Q \subseteq P$ is an \emph{up-set} (respectively, a \emph{down-set}) if given $q \in Q$ and $q \leq p$ (respectively, $q \in Q$ and $p \leq p$) we have $p \in Q$.
			\item[(ii)] A \emph{principal up-set} (respectively, \emph{principal down-set}) \emph{defined by an element $p \in P$} is an up-set (respectively, a down-set) of the form $P_{p \leq } = \{ q \in P \mid p \leq q\}$ (respectively, $P_{\leq p } = \{ q \in P \mid q \leq p \}$).
		\end{enumerate}
	\end{definition}
	
	\begin{definition}
		If $(A, \leq)$ and $(B, \leq')$ are ordered sets, then we say that a map $f: A \to B$ is \emph{order-preserving} if for all $x, y \in A$, $x \leq y$ implies $f(x) \leq' f(y)$.
	\end{definition}
	The following result is proven in \cite{Blyth_2005}*{p.6, Theorem 1.3}.
	\begin{theorem}
		If $(A, \leq)$ and $(B, \leq')$ are ordered sets, then the following conditions about a map $f: A \to B$ are equivalent.
		\begin{enumerate}
			\item[(1)] For any $y \in B$ there is an $x \in A$ such that $f^{-1}(B_{\leq' y}) = A_{\leq x}$.
			\item[(2)] The map $f$ is order-preserving and there is an order-preserving map $g: B \to A$ such that $\id_A \leq g \circ f$ and $f \circ g \leq' \id_B$.
		\end{enumerate}
	\end{theorem}
	
	\begin{definition}\label{def_poset_topology}
		The \emph{coarse lower topology} (respectively \emph{coarse upper topology}) on an ordered set $(P, \leq)$ is a topology that has the principal up-sets , (respectively, the principal down-sets) as a subbasis of closed sets.
	\end{definition}
	The coarse upper topology coincides with the coarse lower topology for the inverse order of $P$.
	
	\begin{definition}\label{def_congruence}
		Let $S$ be a semiring. An equivalence relation $\rho \subseteq S \times S$ is said to be a \emph{congruence relation} if given any $(r_1,r_2), (s_1,s_2) \in \rho$ we have $(r_1 + s_1, r_2 + s_2), (r_1 s_1, r_2 s_2) \in \rho$.
	\end{definition}
	
	The set $\Cong(S)$ of all congruence relations on $S$, can be seen as a poset with respect to the set inclusion. Let $Y$ be any non-empty subset of $\Cong(S)$. Define its greatest lower bound as the element $\wedge Y$ of $\Cong(S)$ by $(r_1,r_2) \in \wedge Y$ if $(r_1,r_2) \in \rho$ for all $\rho \in Y$ and its least upper bound as the element $\vee Y$ of $\Cong(S)$ defined by $(r_1, r_2) \in \vee Y$ if there exists a finite sequence of elements $r_1 = s_0, s_1, \cdots, s_n = r_2$ of $S$ and elements $\rho_1, \cdots, \rho_n \in Y$ such that $(s_{i-1},s_i) \in \rho_i$ for all $1 \leq i \leq n$. The set $\Cong(S)$ forms a complete lattice. In particular, the congruence relation generated by a set $X\subset S\times S$ always exists. We write $S/T$ for the set of equivalence classes. This set has a structure of semiring coming from the one of $S$. 
	
	\begin{lemma}\label{Tau_Congruence_is_k-ideal}
		If $T \in Cong\left( S \right)$, then $\mathfrak{r}(T) := \{ a \in S : (a,0)\in T \} \in
		id_k(S)$.
	\end{lemma}
	\begin{proof}
		That $\mathfrak{r}(T) \in \id(S)$ holds is because $T$ is a semiring. If $\overline{x}$ is the image of $x \in S$ in $S/T$, then for $a,b \in \mathfrak{r}(T)$ and $c \in S$ with $a + c = b$, we have $\overline{c} = \overline{a} + \overline{c} = \overline{a + c} = \overline{b} = \overline{0}$. Hence, $(c,0) \in T$, which concludes.
	\end{proof}
	Endowing $(Cong(S), \subseteq)$ and $(Id_k(S), \subseteq)$ with their coarse lower topologies, we consider the diagram
	\begin{equation}\label{eq:thm_ret_cong}
		\begin{tikzcd}
			\left( Id_k\left( S \right), \subseteq \right) \arrow[rr, "\mathfrak{c}", shift left=4] &  & \left( Cong\left( S \right), \subseteq \right) \arrow[ll, "\mathfrak{r}", shift left=2]
		\end{tikzcd}
	\end{equation}
	where $\frakc$ sends an $I \in Id_k\left( S \right)$ to the semiring congruence generated by $\{ (a, 0) \mid a \in I \}$.
	
	\begin{theorem}\label{Retraction_Theorem}
		The map $\mathfrak{r}$ of \eqref{eq:thm_ret_cong} is a retraction of topological spaces.
	\end{theorem}
	\begin{proof}
		From the definitions, if $I \in Id_k\left( S \right)$ then $(\mathfrak{r} \circ \frakc)(I) = I$ and if $T \in Cong(S)$, then $(\mathfrak{c} \circ \mathfrak{r})(T) \subset T$. In particular, $\mathfrak{c}$ is injective, and since $I\subseteq J$ in $\id_k(S)$ if and only if $\frakc(I)\subseteq \frakc(J)$ in $Cong(S)$, then $\mathfrak{c}$ is a homeomorphism onto its image. Finally, to prove the continuity of $\mathfrak{r}$ we can restrict to subbasic closed sets. If $I \in \id_k(S)$, then $\mathfrak{r}^{-1}(\id_k(S)_{I \subseteq}) = \Cong(S)_{\mathfrak{c}(I) \subseteq}$. We conclude by \cite{Giansiracusa_Giansiracusa_2016}*{Proposition 2.4.4 (3)}.
	\end{proof}
	
	\begin{convention}\label{convention_quotients}
		We now identify $\id_k(S)$ with its image $\mathfrak{c}(Id_k(S))\subset Cong(S)$ and write $S/I$ for $S/\frakc(I)$. \end{convention}
	
	\begin{corollary}\label{VI_is_SpeckR_over_I}
		For $I\in Id_k(S)$, there is prime-ideal preserving set bijection
		$$
		\id_k(S/I) \to \{J\in Id_k(S) \mid I\subset J\}.
		$$
	\end{corollary}
	\begin{proof}
		Consider the semiring morphism $\pi: S\to S/I$. If $J$ is a (prime) $k$-ideal of $S/I$, then $\pi^{-1}(J)$ is a (prime) $k$-ideal of $S$ containing $I$. This gives the inclusions
		$Id_k(S/I) \subseteq \{J\in Id_k(S) \mid I \subseteq J\}$ and $\Spec_k(S/I) \subseteq {V_k(I)}$.
		Conversely, Theorem \ref{Retraction_Theorem} gives a bijection that sends a $k$-ideal $J$ of $S$ containing $I$ to the $k$-ideal $\pi(J)$ of $S/I$. Verifying that if $J$ is prime, implies that $\pi(J)$ is prime, is straightforward.
	\end{proof}
	
	We now topologically compare ideals and $k$-ideals of a semiring $S$. Consider the diagram
	\begin{equation}\label{eq:thm_ret_ideal}
		\begin{tikzcd}
			(Id_k\left( S \right),\subseteq) \arrow[rr, "\mathfrak{i}", shift left=4] &  &(Id\left( S \right),\subseteq)  \arrow[ll, "\mathfrak{j}", shift left=2],
		\end{tikzcd}
	\end{equation}
	where $\mathfrak{i}$ is the inclusion map and $\mathfrak{j}:\id(S)\to\id_k(S)$ is the surjection that sends an ideal of $S$ to its $k$-closure.
	
	\begin{theorem}\label{Retraction_Theorem_ideales}
		If we endow the posets $(Id(S), \subseteq)$ and $(Id_k(S), \subseteq)$ with the coarse upper topology, then the map $\mathfrak{j}$ from \eqref{eq:thm_ret_ideal} is a retraction of topological spaces.
	\end{theorem}
	\begin{proof}
		The map $\mathfrak{i}$ is visibly continuous. It remains to show that $\mathfrak{j}$ is continuous which can be verified on the subbasic closed sets. Observe that $I \subseteq \mathfrak{j}(I)$, that if $I\subseteq J$ then $\mathfrak{j}(I)\subseteq \mathfrak{j}(J)$, and finally $\mathfrak{j}(\mathfrak{j}(I))=\mathfrak{j}(I)$. For each $J \in Id(S)$, we have $J\subseteq (\mathfrak{i} \circ \mathfrak{j})(J)$, and $I= (\mathfrak{j} \circ \mathfrak{i})(I)$ for each $I \in Id_k(S)$. Hence, by \cite{Blyth_2005}*{Theorem 1.3}, the inverse image under $\mathfrak{j}$ of a subbasic closed set is also a subbasic closed set.
	\end{proof}
	
	Specializing Theorem \ref{Retraction_Theorem} to $S = \id(R)^c$ identifies a $k$-ideal $I$ of $\id(R)^c$ with the semiring congruence $\mathfrak{c}(I)$.
	
	\begin{theorem}\label{cor_sch_str}
		For $T \in \id(R)$ with $I := \langle u_R(T) \rangle_k \in \id(R)$, the following diagram of topological spaces
		\[
		\begin{tikzcd}
			Id_k\left( \id^c(R/T)\right) \arrow[r, "\langle u_R(\cdot) \rangle_k"]               & Id_k(\id(R)^c/I)                            \\
			\mathrm{Spec}_k\left(\id^c(R/T)\right) \arrow[r] \arrow[u, hook] & \mathrm{Spec}_k(\id(R)^c/I), \arrow[u, hook]
		\end{tikzcd}
		\]
		commutes, where the horizontal maps are homeomorphisms and the vertical maps are continuous injections. In particular, the bottom horizontal map is the composition of homeomorphisms
		\[
		\begin{tikzcd}
			\mathrm{Spec}_k(\id^c(R/T)) \arrow[r] & \mathrm{Spec}(R/T) \arrow[r] & V(T) \arrow[r] & V_k(I) \arrow[r] & \mathrm{Spec}_k(\id(R)^c/I).
		\end{tikzcd}
		\]
	\end{theorem}
	\begin{proof}
		The map $\langle u_R(\cdot) \rangle_k$ is a homeomorphism with respect to the coarse upper topology. Note also that the set inclusion $\Spec_k(\id(R)^c)\hookrightarrow\id_k(\id(R)^c)$ is continuous. We have the following diagram of ordered set bijections:
		\begin{equation}
			\xymatrix{
				\{J\in Id_k(\id(R)^c) \mid I\subseteq J\}\ar[rr]^{\alpha}\ar[d]^-{\beta}&&Id_k(\id(R)^c/I)\\
				\{T'\in \id(R) \mid T \subseteq T' \}\ar[r]^-{\gamma}&Id(R/T)\ar[r]^-{\delta}&Id_k(\id^c(R/T)),\\}
		\end{equation}
		where $\alpha$ comes from Corollary  \ref{VI_is_SpeckR_over_I}, $ \beta$ and $\delta$ come from Theorem \ref{Correspondence_Modules_over_a_Ring} and $\gamma$ is the correspondence theorem for commutative rings. These maps all preserve primes, hence so does their composition.
	\end{proof}
	
	\begin{corollary}\label{Corollary_of_Correspondence}
		Let $R$ be a ring and let $I\subset R$ be an ideal. Let $R'=\id(R)^c$ and let $I'\subset R'$ be the corresponding $k$-ideal. Equipped with the Zariski topology, $\Spec_k(\id(R/I)^c)$ and $\Spec_k(R'/I')$ are homeomorphic.
	\end{corollary}
	\begin{proof}
		This is because the correspondences of Theorem \ref{cor_sch_str} respect prime ideals.
	\end{proof}
	
	\section{Idempotentization
		of Sheaves and Schemes}\label{Section_Idempotentization}
	
	\subsection{Summary}
	This section is about the idempotentization process and is structured as follows.
	In Subsection \ref{Subsection_Step1}, we define functors $\Phi$ that send a presheaf $\calF$ of rings or modules over a scheme $X$, to a presheaf $\Phi(\calF)$ of $\B$-algebras or of $\id^c(\calO_X)$-semimodules over $X$. Then, in Subsection \ref{Subsection_Step2}, we put structures of idempotent semirings and semimodules on the stalks of $\Phi(\calF)$. Since $\Phi(\calF)$ is not a sheaf in general, we sheafify it in Subsection \ref{Subsection_Step3}. For technical reasons, we now restrict to affine schemes $X = \Spec(R)$. In Subsection \ref{Subsection_Step4}, we use the homeomorphism $\widetilde{u_R} : \Spec(R) \to \Spec_k(\id(R)^c)$ and push forward the sheafified presheaf $\Phi(\calF)^\#$ to get a sheaf on $\Spec_k(\id(R)^c)$. This is the $R$-idempotentization of the presheaf $\calF$. Specializing to Noetherian rings $R$, we show that in this context $\Phi(\calO_X)$ is already a sheaf and that the $R$-idempotentization of $\calO_X$ preserves global sections. We also prove that $R$-idempotentization is functorial in $R$. We end this section by proving in Subsection \ref{Subsection_Step5} that idempotentization can be defined with respect to a covering by distinguished affine open subsets of a ring $R$ and prove the resulting sheaves do not depend on the choice of covering of $X$.
	
	The category of idempotent semirings is equivalent to the category $\mathbb{B}-\Alg$ of $\mathbb{B}$-algebras. We denote by $\SMod^{\text{idem}}$ the category of idempotent semimodules over arbitrary idempotent semirings. We write $PSh(X,\Ring)$ and $ PSh(X,\B-\Alg)$ (respectively $Sh(X,\Ring)$ and $Sh(X,\B-\Alg)$) for the categories of presheaves (respectively sheaves) on $X$ with values in $\Ring$ and in $\B-\Alg$. We also write $PMod(\calO_X)$ and $Mod(\calO_X)$ for the categories of presheaves and sheaves of $\calO_X$-modules on $X$.
	
	
	\subsection{From Commutative Sheaves to Semicommutative Idempotent Presheaves}\label{Subsection_Step1}
	Let $\mathrm{Open}(X)$ be the poset of open subsets of a topological space $X$ with respect to set inclusion. All sheaves we consider are such with respect to a Zariski topology.
	
	\begin{definition}\label{Definition_Idempotentization_Presheaf}
		Let $X$ be a scheme and let $\calF$ be a presheaf of sets on it. For each $U \in \mathrm{Open}(X)$, set
		$$
		\Phi( \calF )(U) :=
		\begin{cases}
			\text{semiring } \id^c(\calF(U)) &\text{ if } \calF \in PSh(X,\Ring),\\
			\id^c(\calO_X(U))-\text{semimodule } \Mod^c(\calF(U)) &\text{ if } \calF \in PMod(\calO_X),
		\end{cases}
		$$
		and for each $V \subseteq U$ in $\mathrm{Open}(X)$, if $\rho_{U,V} : \calF(U) \to \calF(V)$ denotes the restriction map, set
		$$
		\Phi(\rho_{U,V}) :=
		\begin{cases}
			\id^c(\rho_{U,V})  &\text{ if } \calF \in PSh(X,\Ring),\\
			\Mod^c(\rho_{U,V}) &\text{ if } \calF \in PMod(\calO_X).\\
		\end{cases}
		$$
		This defines a presheaf on $X$,
		$$
		\Phi(\calF) : \mathrm{Open}(X)^{\text{op}} \to \calC : V \mapsto \Phi(\calF(V)),
		$$
		where $\calC$ is the category
		$$
		\calC :=
		\begin{cases}
			\B-\Alg &\text{ if } \calF \in PSh(X,\Ring),\\
			\SMod^{\text{idem}} &\text{ if } \calF \in PMod(\calO_X).\\
		\end{cases}
		$$
	\end{definition}
	
	From Definition \ref{Definition_Idempotentization_Presheaf}, we obtain a functor $\Phi: \mathcal{A} \to \mathcal{B}$, with
	$$
	\Phi :=
	\begin{cases}
		\id^c &\text{ if } (\mathcal{A}, \mathcal{B} ) = ( PSh(X,\Ring), PSh(X,\B-\Alg)),\\
		\Mod^c &\text{ if } (\mathcal{A}, \mathcal{B} ) = ( PMod(\calO_X), PSMod(\id^c(\calO_X))),\\
	\end{cases}
	$$
	where $PSMod(\id^c(\calO_X))$ is the category of presheaves of $\id^c(\calO_X)$-semimodules on $X$.
	\begin{remark}\label{Remark_Categories_Well_Defined}
		The category $PSMod(\id^c(\calO_X))$ is well defined because $\id^c(\calO_X)$ is a presheaf of idempotent semirings and for each inclusion of subsets $V \subseteq U$ in $\mathrm{Open}(X)$, following the diagram is commutative
		\[
		\begin{tikzcd}
			\id^c(\mathcal{O}_X(U)) \times \Mod^c(\mathcal{F}(U)) \arrow[r] \arrow[d] & \Mod^c(\mathcal{F}(U)) \arrow[d] \\
			\id^c(\mathcal{O}_X(V)) \times \Mod^c(\mathcal{F}(V)) \arrow[r]           & \Mod^c(\mathcal{F}(V)),       
		\end{tikzcd}
		\]
		where the vertical maps are restriction maps and the horizontal ones are multiplication maps.
	\end{remark} 
	
	\subsection{Stalks of Idempotent Presheaves}\label{Subsection_Step2}
	
	A presheaf $\Phi(\calF)$ on $X$ needs not to be a sheaf in general and so we want to consider its sheafification. Because of this, we now define the stalks of $\Phi(\calF)$ and then show that the functor $\Phi$ commutes with stalks in Remark \ref{Phi_Commutes_with_Stalks}. The verification of the following lemma is immediate.
	
	\begin{lemma}\label{N-Alg_and_B-Alg_are_Complete_and_Cocomplete}
		The categories $\N-\Alg$ and $\mathbb{B}-\Alg$ are complete and cocomplete.
	\end{lemma}
	
	\begin{definition}\label{Definition_Stalk_fg}
		Let $X$ be a scheme. For each $x$ in $X$ and each $\calF$ in $ PSh(X,\Ring)$, respectively in $PMod(\calO_X)$, we define the \textit{stalk of $\Phi(\calF)$ at $x$} to be the colimit
		$$
		\Phi(\calF)_x := \varinjlim_{x \in U \in \mathrm{Open}(X)^{\text{op}}} \Phi(\calF)(U).
		$$
	\end{definition}
	\begin{proposition}\label{Stalk_fg_Commutes}
		Let $X$ be a scheme. For each $x \in X$ and each $\calF$ in $ PSh(X,\Ring)$, respectively in $PMod(\calO_X)$, the set morphism
		$$
		\phi_x: \Phi(\calF)_x \to \Phi(\calF_x)
		$$
		given by sending a class $[(U, \sum_{i=1}^m \calF(U) \cdot s_i]$, with $U \subseteq X$ an open subset containing $x$ and sections $s_1, \cdots, s_m$ of $\calF(U)$, to $\sum_{i=1}^m \calF_x \cdot s_{i,x}$, is a bijection.
	\end{proposition}
	\begin{proof}
		It suffices to prove the proposition for the generators of the finitely generated ideals and submodules, respectively, as the arguments extend linearly (but with heavier notation). The proof is similar in both cases. We provide the details for $\calF \in PSh(X,\Ring)$ and leave the proof of the other statement to the reader.
		\item[(i)] The map $\phi_x$ is well-defined:
		Suppose that $[(U, \calF(U) \cdot s)] = [(V, \calF(V) \cdot t)]$. Then, there exists an open set $W$ contained in $U \cap V$ and containing $x$ such that
		$\calF(W) \cdot s|_W = \calF(W) \cdot t|_W$. Since, $s_x = (s|_W)_x$ and $t_x = (t|_W)_x$, then $\phi_x\left( [(U, \calF(U) \cdot s)] \right) = \phi_x\left( [(V, \calF(V) \cdot t)] \right)$.
		
		The map $\phi_x$ is surjective:
		Let $\calF_x \cdot \sigma$ for some $\sigma \in \calF_x$. Then, there exists an open subset $U$ of $X$ containing $x$ and a section $s \in \calF(U)$ such that $s_x = \sigma$ and so $\calF(U) \cdot s$ is a preimage of $\calF_x \cdot \sigma$ by $\phi_x$.
		
		The map $\phi_x$ is injective:
		Suppose that
		\begin{equation}\label{Equality_of_Stalks}
			\phi_x\left([(U, \calF(U) \cdot r)]\right) = \phi_x\left([(V, \calF(V) \cdot s)]\right) = \calF_x \cdot \sigma
		\end{equation}
		for some $\sigma \in \calF_x$. The element $\calF_x \cdot \sigma$ is represented on some open subset $W$ of $X$ containing $x$ by a section $t \in \calF(W)$. Therefore, $\calF_x \cdot t_x = \phi_x\left( [W, \calF(W) \cdot t ] \right) = \calF_x \cdot \sigma$.
		To conclude, we need to find an open subset $A \subseteq U \cap V \cap W$ containing $x$ such that $\calF(A) \cdot r|_A = \calF(A) \cdot s|_A = \calF(A) \cdot t|_A$. By \eqref{Equality_of_Stalks}, there exist $a, b, c, d \in \calF_x$ such that $a t_x = r_x$, $b t_x = s_x$, $c r_x = t_x$ and $d s_x = t_x$. Hence, there exists an open $A \subseteq U \cap V \cap W$ such that
		$\calF(A) \cdot r|_A, \calF(A) \cdot s|_A\subset \calF(A) \cdot t|_A$ as well as $\calF(A) \cdot t|_A \subseteq \calF(A) \cdot r|_A$ and $\calF(A) \cdot t|_A \subset \calF(A) \cdot s|_A$, which concludes.
	\end{proof}
	
	\begin{convention}\label{Convention_Staks}
		We use the inverse $\phi_x^{-1}$ of the bijection $\phi_X$ of Proposition \ref{Stalk_fg_Commutes} to make $\Phi(\calF)_x$ into an idempotent semiring when $\calF \in PSh(X,\Ring)$. If $\calF \in PMod(\calO_X)$, we denote by $\psi_x^{-1}$ the corresponding inverse bijection $\psi_x^{-1} : \Mod^c(\calF_x) \to \Mod^c(\calF)_x$ of Proposition \ref{Stalk_fg_Commutes}. With this map, $\Phi(\calF)_x$ is an idempotent monoid. Now, let $\mu(x)$ be the multiplication map for the $\id^c(\calO_{X,x})$-semimodule structure on $\Mod^c(\calF_x)$. Then, $\Phi(\calF)_x$ becomes an idempotent semimodule over $\id^c(\calF)_x$ by choosing for multiplication the map $\mu_x := \psi_x^{-1} \circ \mu(x) \circ (\phi_x, \psi_x)$. This is the unique map making the following diagram commutes.
		\begin{equation}\label{Presheaf_Induced_Semimodule_Structure}
			\begin{tikzcd}
				{\id^c(\mathcal{O}_{X,x}) \times \Mod^c(\mathcal{F}_x)} \arrow[rr, "{(\phi_x^{-1},\psi_x^{-1})}"] \arrow[d, "\mu(x)"'] &  & \id^c(\mathcal{O}_X)_x \times \Mod^c(\mathcal{F})_x \arrow[d, "\mu_x", dashed] \\
				\Mod^c(\mathcal{F}_x) \arrow[rr, "\psi_x^{-1}"]                                                                                &  & \Mod^c(\mathcal{F})_x                                                                 
			\end{tikzcd}
		\end{equation}
	\end{convention}
	
	\begin{remark}\label{Phi_Commutes_with_Stalks}
		Let $X$ be a scheme and $x \in X$. Given $\calF, \calG$ in $PSh(X,\Ring)$, respectively in $ PMod(\calO_X)$, and a morphism $f: \calF \to \calG$, we obtain a morphism $\Phi(f) : \Phi(\calF) \to \Phi(\calG)$ in $PSh(X,\B-\Alg)$, respectively in $PSMod(\id^c(\calO_X))$, where $\Phi(f)$ is defined for each open subset $U$ of $X$ by $\Phi(f)(U) := \Phi(f(U))$. From Proposition \ref{Stalk_fg_Commutes}, the diagrams
		\[
		\begin{tikzcd}
			{PSh(X,\mathrm{Ring})} \arrow[r, "\id^c"] \arrow[d, "(\cdot)_x"'] & {PSh(X,\mathbb{B}-\Alg)} \arrow[d, "(\cdot)_x"] & PMod(\mathcal{O}_X) \arrow[r, "\Mod^c"] \arrow[d, "(\cdot)_x"'] & PSMod(\id^c(\mathcal{O}_X) \arrow[d, "(\cdot)_x"]        \\
			\mathrm{Ring} \arrow[r, "\id^c"]                                & \mathbb{B}-\Alg                               & {\mathcal{O}_{X,x}-\mathrm{mod}} \arrow[r, "\Mod^c"]          & {\id^c(\mathcal{O}_{X,x})-\mathrm{SMod}^{\text{idem}}}
		\end{tikzcd}
		\]
		commute. We express this by saying that \textit{the functor $\Phi$ commutes with stalks}.
	\end{remark}
	
	\subsection{Sheafification of Idempotent Presheaves}\label{Subsection_Step3}
	
	We now sheafify the presheaf $\Phi(\calF)$.
	
	\begin{propositiondefinition}
		Let $X$ be a scheme and let $\calF \in \{ PSh(X,\Ring), PMod(\calO_X) \}$. The \textit{sheafification} $\Phi(\calF)^\sharp$ of $\Phi(\calF)$ is the presheaf defined, for each open subset $U$ of $X$, by
		$$
		\Phi(\calF)^\sharp(U) := \left\{ s : U \to \coprod_{\pp \in U} \Phi(\calF)_{\pp} \text{ such that } (*) \text{ holds}\right\}
		$$
		where (*) is that for each $s \in \Phi(\calF)(U)$ and each $\pp \in U$:
		\begin{enumerate}
			\item[(i)] $s(\pp) \in \Phi(\calF)_\pp$, and
			\item[(ii)] there is an open neighbourhood $V_\pp$ of $\pp$ in $U$ together with an element $t \in \Phi(\calF)(V_\pp)$, such that for all $\mathfrak{q} \in V_\pp$, the image $t_\mathfrak{q}$ of $t$ in $\Phi(\calF)_{\mathfrak{q}}$ equals $s(\mathfrak{q})$.
		\end{enumerate}
		
		For each inclusion of open subsets $V \subset U$ in $\mathrm{Open(X)}$ the projection map
		$$
		\prod_{u \in U} \Phi(\calF)_u \to \prod_{v \in V} \Phi(\calF)_v
		$$
		sends elements of $\Phi^\sharp(\calF)(U)$ into $\Phi^\sharp(\calF)(V)$. Taking these maps for restriction maps, we turn $\Phi(\calF)^\sharp$ into a presheaf of sets in the usual way (see \cite{stacks-project}*{Tag 007Y}) and the construction then satisfies the universal property of sheafification for sheaves of sets (see \cite{stacks-project}*{Tag 0080}).
	\end{propositiondefinition}
	This gives sheaves of sets. We now endow them with suitable idempotent algebraic structures, namely idempotent semirings and idempotent semimodules.
	\begin{proposition}\label{Proposition_Sheaf_With_Appropriate_Idempotent_Structure}
		Let $X$ be a scheme and let $\calF$ in $PSh(X,\Ring)$, respectively in $PMod(\calO_X)$. The sheaf of sets $\Phi(\calF)^\sharp$ can be turned into a sheaf of idempotent semirings, respectively into a sheaf of $\id^c(\calO_X)^\sharp$-modules, via the set of bijections $\{ \phi_x^{-1} : \Phi(\calF_x) \to \Phi(\calF)_x \mid x \in X\}$ and $\{ \psi_x^{-1} : \Mod^c(\calF_x) \to \Mod^c(\calF)_x \mid x \in X \}$ of Proposition \ref{Stalk_fg_Commutes} and Convention \ref{Convention_Staks}.
	\end{proposition}
	\begin{proof}
		Suppose that $\calF \in PSh(X,\Ring)$. By \cite{stacks-project}*{Tag 007Z} and Proposition \ref{Stalk_fg_Commutes}, we have bijections
		\begin{equation}\label{Stalk_Induced_Idempotent_Semiring_Structure}
			\id^c(\calF_x) \xrightarrow{\phi_x^{-1}} \id^c(\calF)_x \xrightarrow{\simeq} \id^c(\calF)_x^\sharp,
		\end{equation}
		which endow $\id^c(\calF)_x^\sharp$ with a structure of idempotent semiring. Now, Proposition \ref{Correspondence_Modules_over_a_Ring} gives a semiring isomorphism from $\id(\calF_x)$ to $\id_k(\id^c(\calF_x))$ . Since the ring $\calF_x$ is local, its unique maximal ideal corresponds to a unique maximal $k$-ideal of $\id^c(\calF_x)$. The semiring $\id^c(\calF_x)$ is therefore local.
		
		Next, for each open subset $U$ of $X$, we endow the set $\id^c(\calF)^\sharp(U)$ with a structure of idempotent semiring. For each $s = (s_u)_{u \in U}, t = (t_u)_{u \in U} \in \id^c(\calF)^\sharp(U)$, we set
		\begin{equation}\label{Equation_Operations}
			s + t := (s_u + t_u)_{u \in U} \text{ and } st := (s_u \cdot t_u)_{u \in U}.
		\end{equation}
		With \eqref{Equation_Operations}, $\id^c(\calF)^\sharp(U)$ is an idempotent semiring since this is true entrywise.
		
		Finally, we show that, for each inclusion $V \subseteq U$ of open subsets of $X$, the restriction map from $\id^c(\calF)^\sharp(U)$ to $\id^c(\calF)^\sharp(V)$ is a semiring morphism. Let $u \in U$. There exist inside $U$ open neighbourhoods $V_1$ and $V_2$ of $x$ such that there is a pair $(V_1,\sigma)$ with $\sigma = \langle s_1, \cdots, s_m \rangle \in \id^c(\calF(V_1))$ representing all stalks $s_y$ with $y \in V_1$ and a pair $(V_2,\tau)$ with $\tau = \langle t_1, \cdots, t_n \rangle \in \id^c(\calF(V_q))$ representing all stalks $t_z$ with $z \in V_2$. Let $V := V_1 \cap V_2$ of $x$. Then, $\sigma|_V,\tau|_V \in \id^c(\calF(V))$, $\sigma|_{V}+\tau|_V = \langle s_i|_V, t_j|_V, 1 \leq i \leq m, 1 \leq j \leq n \rangle$ and $\sigma|_{V} \cdot \tau|_V = \langle s_i|_V \cdot t_j|_V, 1 \leq i \leq m, 1 \leq j \leq n \rangle$. So, for each $v \in V$ we have
		$$
		(\sigma|_{V}+\tau|_V)_v = \langle s_{i,v}, t_{j,v}, 1 \leq i \leq m, 1 \leq j \leq n \rangle = (\sigma|_{V})_v + (\tau|_V)_v = s_v + t_v
		$$
		and
		$$
		(\sigma|_{V}\tau|_V)_v =\langle s_{i,v} \cdot t_{j,v} | 1 \leq i \leq m, 1 \leq j \leq n \rangle =(\sigma|_{V})_v \cdot (\tau|_V)_v =s_v \cdot t_v.
		$$
		
		Thus, the restriction map is a semiring morphism.
		
		Now suppose that $\calF \in PMod(\calO_X)$. The corresponding bijections
		\begin{equation}\label{Stalk_Induced_Idempotent_Semimodule_Structure}
			\Mod^c(\calF_x) \xrightarrow{\psi_x^{-1}} \Mod^c(\calF)_x \xrightarrow{\simeq} \Mod^c(\calF)_x^\sharp
		\end{equation}
		endow $\Mod^c(\calF_x)^\sharp$ with a structure of idempotent monoid. Combining \eqref{Presheaf_Induced_Semimodule_Structure} with Equations \eqref{Stalk_Induced_Idempotent_Semiring_Structure} and \eqref{Stalk_Induced_Idempotent_Semimodule_Structure}, we endow $\Mod^c(\calF)^\sharp_x$ with a structure of idempotent $\id^c(\calO_X)^\sharp_x$-semimodule. Given an open subset $U$ of $X$ and elements $s = (s_u)_{u \in U}$ of $\id^c(\calO_X)^\sharp(U)$ and $m = (m_u)_{u \in U}$ and $n = (n_u)_{u \in U}$ of $\Mod^c(\calF)^\sharp(U)$, we set
		$$
		m + n := (m_u + n_u)_{u \in U} \text{ and } sm := (s_u \cdot m_u)_{u \in U}.
		$$
		These operations give to $\Mod^c(\calF)^\sharp(U)$ a structure of idempotent $\id^c(\calO_X)^\sharp(U)$-semimodule. For each inclusion $V \subseteq U$ in $\mathrm{Open}(X)$, we see that the following diagram commutes.
		\[
		\begin{tikzcd}
			\id^{c,\sharp}(U)\times \Mod^{c,\sharp}(U) \arrow[rr] \arrow[d] &  & \Mod^{c,\sharp}(U) \arrow[d] \\
			\id^{c,\sharp}(V)\times \Mod^{c,\sharp}(V) \arrow[rr]           &  & \Mod^{c,\sharp}(V)          
		\end{tikzcd}
		\]
	\end{proof}
	\begin{remark}\label{Sheafification_is_a_Functor}
		The above constructions give functors
		$$
		\sharp : PSh(X,\mathbb{B}-\Alg) \to Sh(X, \mathbb{B}-\Alg) \text{ and } \sharp : PSMod(\id^c(\calO_X)) \to SMod(\id^c(\calO_X)^\sharp).
		$$
	\end{remark}
	
	\subsection{Idempotentization with respect to an affine scheme}\label{Subsection_Step4}
	
	Let $R$ be a ring. Recall that we have a functor $\id^c_R$ from $Sh(\Spec(R),\Ring)$ to $PSh(\Spec(R),\B-\Alg)$ and a functor $\sharp_R$ from $PSh(\Spec(R), \B-\Alg)$ to $Sh(\Spec(R),\B-\Alg)$. Moreover, the universal valuation $u_R$ induces a homeomorphism $\widetilde{u_R}$ from $\Spec(R)$ to $\Spec_k(\id(R)^c)$ which in turns induces a pushforward functor $\widetilde{u_R}_*$ which sends a sheaf $\calF$ of $\B$-algebras on $\Spec(R)$ to the sheaf $\widetilde{u_R}_*\calF$ of $\B$-algebras on $\Spec_k(\id(R)^c)$ whose sections on an open subset $V$ are given by $\calF(\widetilde{u_R}^{-1}(V))$. We write $\idm_R$ for the composite functor $\widetilde{u_R}_* \circ \sharp_R \circ \id_R^c$ from $Sh(\Spec(R),\Ring)$ to $Sh(\Spec_k(\id(R)^c),\B-\Alg)$.
	\begin{definition}
		The $R$-idempotentization of the affine scheme $(\Spec(R), \calO_{\Spec(R)})$ is the pair with topological space
		$\Spec_k(\id(R)^c)$ and structure sheaf $\calO_{\idm(\Spec(R))} := \idm_R(\calO_{\Spec(R)})$.
	\end{definition}
	The case where $R$ is Noetherian is particularly interesting since the structure sheaf recovers the global sections as we now show.
	The next lemma essentially holds because we consider a sheaf with values in a type of algebraic structure (in the sense of \cite{stacks-project}*{Tag 007M}) that is defined over a Noetherian scheme.
	\begin{lemma}\label{Lemma_Favourable_Conditions_Sheaf}
		Let $X = \Spec(R)$ be an affine scheme with $R$ a Noetherian ring. Then the presheaf of idempotent semirings $\id^c(\calO_X)$ on $X$ is a sheaf.
	\end{lemma}
	\begin{proof}
		Since $\mathbb{B}-\Alg$ together with the forgetful functor to $\Set$ is a type of algebraic structure then we can show the sheaf condition on a basis of the topological space $X$ (see \cite{stacks-project}*{Tag 009R}). We choose the basis $\mathcal{B} = \{ D(f), f \in R \}$. Let $D(f) \in \mathcal{B}$ and consider a covering $D(f) = \cup_{i \in I} D(f_i)$ and for each pair $(i,j) \in I \times I$, a covering $D(f_i f_j) = D(f_i) \cap D(f_j) = \cup_{k \in I_{i,j}} D(f_{ijk})$ with $D(f_{ijk}) \in \mathcal{B}$. Consider a collection of sections $\{ s_i \}_{i \in I}$ with $s_i \in \id^c( \calO_X(D(f_i))$ such that for each pair $(i,j) \in I \times I$ and each $k \in I_{i,j}$, we have
		\begin{equation}\label{Local_Equality}
			s_i|_{D(f_{ijk})} = s_j|_{D(f_{ijk})} \text{ in } \id^c(\calO_X(D(f_{ijk}))).
		\end{equation}
		Recall that if $A$ is a commutative ring, then $\id^c(A) = \Mod^c(A, A)$. Since $D(f)$ is an affine Noetherian scheme, the category $\calO_X(D(f))-\Mod$ is equivalent to the category of quasi-coherent $\calO_{D(f)}$-modules (see \cite{stacks-project}*{Tag 01IB}). The data of $\{ s_i \}_{i \in I}$ together with \eqref{Local_Equality} translates into a descent datum of quasi-coherent sheaves. Descent being effective we obtain, by equivalence of categories, a unique ideal $J$ of $\calO_X(D(f))$ whose restriction to each $\calO_{X}(D(f_i))$ is the ideal $s_i$. The corresponding sheaf of $\calO_X$-modules $\widetilde{J}$ on $D(f)$ is then coherent (see \cite{stacks-project}*{Tag 01XZ}). Therefore, $J$ is a finitely generated ideal, concluding the proof.
	\end{proof}
	
	Since $\id(R)^c = \id(R)$, the claimed results amount to show the following.
	\begin{theorem}\label{Theorem_Global_Sections}
		Let $R$ be a Noetherian ring. Then 
		$$
		\calO_{\idm(\Spec(R))}\left( \idm(\Spec(R))\right) = \id(R).
		$$
	\end{theorem}
	\begin{proof}
		By Lemma \ref{Lemma_Favourable_Conditions_Sheaf} and the fact that $\widetilde{u_R} : \Spec(R) \to \Spec_k(\id(R))$ is a homeomorphism, we have
		$$\calO_{\Spec_k(\id(R))}\left( \Spec_k(\id(R))\right) = \tilde{u}_{\Spec(R)*}\left( \id^c\left( \calO_{\Spec(R)}\right)\right)\left( \Spec_k\left( \id(R)\right)\right) = \id^c\left(\calO_{\Spec(R)}\right)\left( \Spec(R) \right) = \id(R).
		$$
	\end{proof}
	
	\begin{definition}
		The $R$-idempotentization of a sheaf of $\calO_{\Spec(R)}$-modules $\calF$ is the sheaf $\widetilde{u_R}_*(\Mod^c(\calF)^\#)$ of $\calO_{\idm(\Spec(R))}$-semimodules over $\Spec_k(\id(R)^c)$.
	\end{definition}
	This defines a functor
	$$
	\widetilde{u_R}_* \circ \#_R \circ \Mod^c : \mathrm{Mod}(\calO_{\Spec(R)}) \to \SMod(\idm(\calO_{\Spec(R)})).
	$$
	
	\begin{example}\label{Example_Sheaf_Reinterpretation_Rump_Yang} The semiring $\mathrm{Frac}(\id(R)^c)^\times$ from Example \ref{Example_Bezout} can turned  into a presheaf on $\Spec(R)$, denoted $\mathrm{Frac}(\id^c(\calO_{\Spec(R)}))^\times$. Its sheafification is denoted $G(\tilde{R})$ in \cite{Rump_Yang_2008}*{p.270} and can be identified, through the homeomorphism $\widetilde{u_R}$, with its idempotentization $\widetilde{u_R}_*G(\widetilde{R})$. On the Zariski topological space of $\ell$-ideals $\mathrm{Sp}(G)$ of $G(R)$, one can define a sheaf $\check{G}$. That sheaf is identified to $G(\tilde{R}$) because of \cite{Rump_Yang_2008}*{Proposition 7}), and hence to our idempotentization $\widetilde{u_R}_* G(\tilde{R})$.
	\end{example}
	
	Idempotentization is functorial in $R$. Indeed, if $f : R_2 \to R_1$ is a ring morphism, then $u_{R_1} \circ f = \id(f)^c \circ u_{R_2}$ (see Subsection \ref{Subsection_Universal_Valuation}). This yields a commutative diagram
	\[
	\begin{tikzcd}
		\mathrm{Spec}(R_1) \arrow[rr, "\mathrm{Spec}(f)"] \arrow[d, "\widetilde{u_{R_1}}"]  &  & \mathrm{Spec}(R_2) \arrow[d, "\widetilde{u_{R_2}}"] \\
		\mathrm{Spec}_k(\mathrm{Id}(R_1)^c) \arrow[rr, "\mathrm{Spec}_k(\mathrm{Id}^c(f))"] &  & \mathrm{Spec}_k(\mathrm{Id}(R_2)^c).                
	\end{tikzcd}
	\]
	Write $\varphi := \Spec(f)$ and $\varphi_k := \Spec_k(\id^c(f))$. For $\psi: \Spec_k(\id(R_1)^c) \to \Spec_k(\id(R_2)^c)$ continuous, a sheaf $\calG$ of $\B$-algebras on the former space is sent to the sheaf of $\B$-algebras $\psi_*\calG$ on the latter, who sections on every open subset $V$ of $\Spec_k(\id(R_2)^c)$ are defined to be $\psi_*\calG(V) = \calG( \psi^{-1}(V))$. Then $\idm(\varphi_*) = \varphi_{k,*}$ is an example of $\psi$ and one verifies that the following diagram commutes.
	\[
	\begin{tikzcd}
		{Sh(\mathrm{Spec}(R_1),\mathrm{Ring})} \arrow[rr, "\mathrm{Idm}_{R_1}"] \arrow[dd, "\varphi_*"] &  & {Sh(\mathrm{Spec}_k(\mathrm{Id}(R_1))^c,\mathbb{B}-\mathrm{Alg})} \arrow[dd, "\mathrm{idm}(\varphi_*)"] \\
		&  &                                                                                                         \\
		{Sh(\mathrm{Spec}(R_2),\mathrm{Ring})} \arrow[rr, "\mathrm{Idm}_{R_2}"]                         &  & {Sh(\mathrm{Spec}_k(\mathrm{Id}(R_2))^c,\mathbb{B}-\mathrm{Alg})}                                      
	\end{tikzcd}
	\]
	\subsection{Idempotentization with Respect to a Covering and Independence}\label{Subsection_Step5}
	
	Let $R$ be a ring and let $\Sigma$ be a generating subset of $R$. Then, the set $\calU_\Sigma := \{ D(f) \mid f \in \Sigma \}$ is an affine open covering of $\Spec(R)$ and this set is in bijection with the set $\calU_{\idm(\Sigma)} := \{ D_k(f) \mid f \in \Sigma \}$ because of \eqref{Equation_Identification_Distinguished_Opens}. To ease the arguments that come next, we identify $D(f)$ with $\Spec(R_f)$ and $D_k(f)$ with $\Spec_k(\id(R)^c_{u_R(f)})$. We now construct functors between descent categories.
	\begin{lemma}
		There is a functor $\rho_\Sigma: \mathrm{Desc}(\calU_\Sigma,\Ring) \to \mathrm{Desc}(\calU_{\idm(\Sigma)},\B-\Alg)$ and a functor $\psi_\Sigma :\mathrm{Desc}(\Mod(\calO_{\calU_{\Sigma}}) \to \mathrm{Desc}(\SMod(\calO_{\calU_{\idm(\Sigma})})$, where $\mathrm{Desc}(\calU_\Sigma,\Ring)$ and $\mathrm{Desc}(\SMod(\calO_{\calU_{\idm(\Sigma})})$ are the categories of descent data of $\calO_{\Spec(R)}$-modules and of $\calO_{\idm(\Spec(R))}$-semimodules, with respect to $\calU_\Sigma$.
	\end{lemma}
	\begin{proof}
		Let $(\calF_f, \phi_{f,g} \mid f,g \in \Sigma)$ be an object of $\mathrm{Desc}(\calU_\Sigma, \Ring)$. Then $\widetilde{u_{R_f}}_*\left( \id^c(\calF_f) \right)^\sharp$ is a sheaf of $\B$-algebras on $D_k(f)$, and an isomorphism  $\phi_{f,g} : \calF_f \mid_{D(fg)} \to \calF_g \mid_{D(fg)}$ of sheaves of rings on $D(fg)$ is transformed into an isomorphism $\idm_{f,g}(\phi_{f,g})$ of sheaves of $\B$-algebras on $D_k(fg)$ satisfying the cocycle condition
		$$
		\idm_{R_{fgh}}(\phi_{g,h}) \circ \idm_{R_{fgh}}(\phi_{f,g}) = \idm_{R_{fgh}}(\phi_{f,h})$$
		over $D_k(fgh)$. Now, suppose given a collection of maps $\theta_f : \calF_f \to \calG_f$ of sheaves of rings over $D(f)$, for varying $f \in \Sigma$ such that for each pair $(f,g)$ of elements of $\Sigma$ such that the restrictions of $\theta_f$ and $\theta_g$ to $D(fg)$ give the same map $\theta_{f,g} : \calF \mid_{D(fg)} \to \calG \mid_{D(fg)}$ of sheaves of rings of $D(fg)$. Then by functoriality, we associate a map $\idm_{R_f}(\theta_f)$ of sheaves of $\B$-algebras on $D_k(f)$ such that for any pair $(f,g)$ of elements of $\Sigma$ the restrictions of  $\idm_{R_f}(\theta_f)$ and $\idm_{R_g}(\theta_g)$ to $D_k(fg)$ give the same map $\idm_{R_{fg}}(\theta_{f,g})$ of sheaves of $\B$-algebras. This describes a functor $\rho_\Sigma$ as desired.
		
		The construction of the second functor is analogous: One first makes the construction for idempotent monoids and then verifies that the linearity of the action is preserved.
	\end{proof}
	
	\begin{remark}\label{On_Types_of_Algebraic_Structures}
		Let $\calC$ be a small category with limits and filtered colimits. Let $F: \calC \to \Set$ be the corresponding forgetful functor. Then, $F$ is immediately seen to be faithful, commuting with limits of $\calC$, commuting with filtered colimits of $\calC$ and reflecting isomorphisms. Therefore, a small category $\calC$ together with the forgetful set functor is a type of algebraic structure (in the sense of \cite{stacks-project}*{Tag 007L}) if and only if $\calC$ has limits and filtered colimits. In light of this, Lemma \ref{N-Alg_and_B-Alg_are_Complete_and_Cocomplete} implies that the pairs $(\B-\Alg,F)$ and $(\N-\Alg,F)$, with $F$ their corresponding forgetful set functor, are types of algebraic structures. Because of this and Remark \ref{Remark_Preserving_Open_Cover}, we have ``descent'' functors
		$$
		Sh(\Spec(R),\Ring) \to \mathrm{Desc}(\calU_\Sigma,\Ring) \text{ and } \mathrm{Desc}(\calU_{\idm(\Sigma)},\B-\Alg) \to Sh( \Spec_k(\id(R)^c), \B-\Alg)
		$$
		and these define equivalences between the corresponding categories (see \cite{stacks-project}*{Tag 00AN}).
	\end{remark}
	
	From Remark \ref{On_Types_of_Algebraic_Structures}, we also have descent functors
	$$
	\Mod(\calO_{\Spec(R)}) \to \mathrm{Desc}(\Mod(\calO_{\calU_{\Sigma}}) \text{ and } \mathrm{Desc}(\SMod(\calO_{\calU_{\idm(\Sigma})}) \to \SMod(\calO_{\idm(\Spec(R))})
	$$
	since idempotent monoids are a type of algebraic structure that is compatible with the structure of idempotent semiring acting on them. The compatibility of the semiring action is immediate.
	
	\begin{definition}
		If $R$ is a ring, then the idempotentization functor of sheaves of rings on $\Spec(R)$ with respect to $\calU_\Sigma$ is defined to be the composite of the above functors
		$$
		Sh(\Spec(R),\Ring) \xrightarrow{\simeq} \mathrm{Desc}(\calU_\Sigma,\Ring) \xrightarrow{\rho_\Sigma} \mathrm{Desc}(\calU_{\idm(\Sigma)},\B-\Alg) \xrightarrow{\simeq} Sh( \Spec_k(\id(R)^c), \B-\Alg).
		$$
		Similarly, the idempotentization functor of sheaves of $\calO_{\Spec(R)}$-modules on $\Spec(R)$ with respect to $\calU_\Sigma$ is defined to be the composite functor
		$$
		\Mod(\calO_{\Spec(R)}) \xrightarrow{\simeq} \mathrm{Desc}(\Mod(\calO_{\calU_{\Sigma}}) \xrightarrow{\psi_\Sigma} \mathrm{Desc}(\SMod(\calO_{\calU_{\idm(\Sigma})}) \xrightarrow{\simeq} \SMod(\calO_{\idm(\Spec(R))}).
		$$
	\end{definition}
	Given $\Sigma_1$ and $\Sigma_2$ two generating sets of $R$, we have the following commutative diagram
	\[
	\begin{tikzcd}
		{Sh(\mathrm{Spec}(R),\mathrm{Ring})} \arrow[d, equal] \arrow[r, "\cong"] & {\mathrm{Desc}(\mathcal{U}_{\Sigma_1},\mathrm{Ring})} \arrow[r, "\rho_{\Sigma_1}"] \arrow[d, "{\rho_{\Sigma_1,\Sigma_2}}"] & {\mathrm{Desc}(\mathcal{U}_{\mathrm{Idm}(\Sigma_1)},\mathbb{B}-\mathrm{Alg})} \arrow[r, "\cong"] \arrow[d, "{\rho_{\mathrm{Idm}(\Sigma_1),\mathrm{Idm}(\Sigma_2)}}"] & {Sh(\mathrm{Spec}_k(\mathrm{Id}(R)^c,\mathbb{B}-\mathrm{Alg})} \arrow[d, equal] \\
		{Sh(\mathrm{Spec}(R),\mathrm{Ring})} \arrow[r, "\cong"']                               & {\mathrm{Desc}(\mathcal{U}_{\Sigma_2},\mathrm{Ring})} \arrow[r, "\rho_{\Sigma_2}"]                                         & {\mathrm{Desc}(\mathcal{U}_{\mathrm{Idm}(\Sigma_2)},\mathbb{B}-\mathrm{Alg})} \arrow[r, "\cong"]                                                                     & {{Sh(\mathrm{Spec}_k(\mathrm{Id}(R)^c,\mathbb{B}-\mathrm{Alg})},}                            
	\end{tikzcd}
	\]
	where the map $\rho_{\Sigma_1,\Sigma_2}$ (respectively $\rho_{\idm(\Sigma_1),\idm(\Sigma_2)}$) sends the descent data of a sheaf of rings on $\Spec(R)$ on $\calU_{\Sigma_1}$ to the corresponding descent data on $\calU_{\Sigma_2}$ (respectively the descent data of a sheaf of $\B$-algebras on $\Spec_k(\id(R)^c)$ on $\calU_{\idm(\Sigma_1)}$ to the corresponding descent data on $\calU_{\idm(\Sigma_2)}$). As a consequence, idempotentization is independent of the choice of covering of $\Spec(R)$ by distinguished open affine subschemes. Analogously, we conclude to the independence of such coverings for sheaves of modules.
	
	\begin{remark}\label{Generalization_of_Construction}
		The reason why we restrict to covering by distinguished open affine subsets is the following. To apply the functor $\id^c$ in the category of descent data, we need that the intersection of affine subschemes is affine. This is not true in general. However, if $X$ is an affine scheme all whose local rings are Noetherian of dimension $0$ or $1$, then its quasi-compact open subsets are affine (see \cite{stacks-project}*{Tag 09N9}). Hence, for a Noetherian affine scheme of dimension $0$ or $1$, idempotentization is independent of any open affine covering of $X$. It would be interesting to extend the idempotentization process to arbitrary open coverings of affine schemes and all schemes in general.  We suspect that extension to locally ringed sites and sheaves of rings or modules over them might be possible (in light of \cite{stacks-project}*{Tags 00W1, 03CX and 04TP}). We leave these considerations for future works and point out that taking inspirations from \cite{Giansiracusa_Giansiracusa_2016},\cite{Giansiracusa_Giansiracusa_2022} and revisiting Yue Ren's question might lead to new developments.
	\end{remark}
	
	\section*{Acknowledgments}
	We would like to acknowledge the Pacific Institute for Mathematical Sciences (PIMS) for their support, the first author as a PIMS postdoctoral fellow at the University of Lethbridge, and the second author as a visiting researcher within the context of the Lethbridge Number Theory and Combinatorics Seminar. The authors also thank the Department of Mathematics and Computer Science of the University of Lethbridge for providing a stimulating work environment where part of this article was written. While a postdoctoral fellow at the Department of Mathematics of the Universit\'e du Luxembourg, the first author benefited from excellent work conditions that allowed him to carry out a thorough revision of the paper and improve its content and exposition. This research was funded in part, by the Luxembourg National Research Fund (FNR), grant reference 501100001866-17921905. For the purpose of open access, and in fulfilment of the obligations arising from the grant agreement, the authors have applied a Creative Commons Attribution 4.0 International (CC BY 4.0) license to any Author Accepted Manuscript version arising from this submission. We would like to thank Pedro Luis del \'Angel Rodr\'iguez, Matthew Baker and Yue Ren for inspiring conversations. We are grateful to Oliver Lorscheid for his observations and questions which helped us to correct some inaccuracies in previous versions of the text and motivated us to detail some arguments that improved the quality of the paper. The first author warmly thanks Kalina Mincheva and Jaiung Jun for providing detailed comments and relevant literature on this paper and also to the former for pointing out an issue with the notion realizable semiring. This feedback helped us to better situate our work within the existing literature and as well as more suitably spotlight our contributions. The second author would like to thank Lara Bossinger, Sebastian Falkensteiner, Mercedes Haiech and Marc Paul Noordman, for useful discussions at the beginning of the project.

\end{document}